\documentclass[reqno,12pt]{amsart}
\usepackage[colorlinks=true, linkcolor=blue, citecolor=blue]{hyperref}

\usepackage{amssymb}
\usepackage{amsmath, graphicx, rotating}
\usepackage{color}
\usepackage{soul}
\usepackage[dvipsnames]{xcolor}

\usepackage{ifthen}
\usepackage{xkeyval}
\usepackage{todonotes}
\usepackage[T1]{fontenc}
\usepackage{lmodern}
\usepackage[english]{babel}

\usepackage{ upgreek }
\usepackage{stmaryrd}
\SetSymbolFont{stmry}{bold}{U}{stmry}{m}{n}
\usepackage{amsthm}
\usepackage{float}

\usepackage{ bbm }
\usepackage{ stmaryrd }
\usepackage{ mathrsfs }
\usepackage{ frcursive }
\usepackage{ comment }

\usepackage{pgf, tikz}
\usetikzlibrary{shapes}
\usepackage{varioref}
\usepackage{enumitem}

\definecolor{rouge}{rgb}{0.7,0.00,0.00}
\definecolor{vert}{rgb}{0.00,0.5,0.00}
\definecolor{bleu}{rgb}{0.00,0.00,0.8}
\usepackage[margin=1in]{geometry}
\newtheorem{theorem}{Theorem}[section]
\newtheorem*{theorem*}{Theorem}
\newtheorem{lemma}[theorem]{Lemma}
\newtheorem{definition}[theorem]{Definition}
\newtheorem{corollary}[theorem]{Corollary}
\newtheorem{proposition}[theorem]{Proposition}

\labelformat{hypothesis}{\textbf{M\kern-0.1mm#1}}

\newtheorem{condition}{Condition}

\newtheorem{conditionA}{A\kern-0.1mm}
\labelformat{conditionA}{\textbf{(A\kern-0.1mm#1)}}

\theoremstyle{definition}

\newtheorem{remark}[theorem]{Remark}

\def \eref#1{\hbox{(\ref{#1})}}

\numberwithin{equation}{section}

\def\geq{\geqslant}
\def\leq{\leqslant}

\def\SS{\mathbb{S}}
\def\HH{\mathbb{H}}
\def\RR{\mathbb{R}}
\def\PP{\mathbb{P}}
\def\EE{\mathbb{E}}

\def\VV{\mathbb{V}}

\def\FF{{\mathcal F}}

\def\de{{\delta}}

\def\de{{\delta}}

\def\tr{{ \hbox{ Tr} }}

\def\e{{\varepsilon}}
\def \eref#1{\hbox{(\ref{#1})}}

\def\EE{\mathbb{ E}}

\def\HS{{\rm HS}}

\begin{document}

\title[Large deviation for two-time-scale stochastic Burgers equation]
{Large deviation for two-time-scale stochastic Burgers equation}

\author{Xiaobin Sun}
\curraddr[Sun, X.]{ School of Mathematics and Statistics, Jiangsu Normal University, Xuzhou, 221116, China}
\email{xbsun@jsnu.edu.cn}

\author{Ran Wang}
\curraddr[Wang, R.]{ School of Mathematics and Statistics, Wuhan University, Wuhan, 430072, China}
\email{rwang@whu.edu.cn}

\author{Lihu Xu}
\curraddr[Xu, L.]
{ Department of Mathematics,
Faculty of Science and Technology,
University of Macau, E11
Avenida da Universidade, Taipa,
Macau, China;
UM Zhuhai Research Institute, Zhuhai, China}
\email{lihuxu@umac.mo}

\author{Xue Yang}
\curraddr[Yang, X.]{School of Mathematics,  Tianjin University, Tianjin, 300350, China}
\email{xyang2013@tju.edu.cn}

\begin{abstract}
A Freidlin-Wentzell type large deviation principle is established for stochastic partial differential equations with slow and fast time-scales, where the slow component is a one-dimensional stochastic Burgers  equation with small noise and the fast component is a stochastic reaction-diffusion equation. Our approach is via the weak convergence criterion developed in \cite{Budhiraja-Dupuis}.
\end{abstract}

\date{\today}

\maketitle

\section{Introduction}
In this paper, we  study  the large deviation  principle (LDP) for  the following stochastic slow-fast system on the interval $[0,1]$:
\begin{equation}
\left\{\begin{array}{l}\label{Equation}
\frac{\partial}{\partial t} X^{\e,\delta}_t(\xi)= \frac{\partial^2}{\partial \xi^2}X^{\e,\delta}_t(\xi)+\frac{1}{2}\frac{\partial}{\partial \xi}\left(X^{\e,\delta}_t\right)^2(\xi)+f\left(X^{\e,\delta}_t, Y^{\e,\delta}_t\right)(\xi),\\
\vspace{2mm}
\hspace{2cm} +\sqrt{\e}\sigma_1\left(X_t^{\e,\delta}\right)(\xi)Q_1^{1/2}\frac{\partial W}{\partial t}(t,\xi), \\
\vspace{2mm}
\frac{\partial }{\partial t}Y^{\e,\delta}_t(\xi)
=\frac{1}{\delta}\big[\frac{\partial^2}{\partial \xi^2} Y^{\e,\delta}_t(\xi)+g\left(X^{\e,\delta}_t, Y^{\e,\delta}_t\right)(\xi)\big]+\frac{1}{\sqrt{\delta}}\sigma_2\left(X_t^{\e,\delta}, Y_t^{\e,\delta}\right)(\xi)Q_2^{1/2}\frac{\partial W }{\partial t}(t,\xi),\\
\vspace{2mm}
X^{\e,\delta}_t(0)=X^{\e,\delta}_t(1)=Y^{\e,\delta}_t(0)
=Y^{\e,\delta}_t(1)=0, \quad t>0,\\
\vspace{2mm}
 X^{\e, \delta}_0=x, \quad Y^{\e,\delta}_0=y,  \end{array}\right.
\end{equation}
where $\e >0,\delta=\delta(\e)>0$ are  small parameters
describing the ratio of time scales between the slow component $X^{\e, \delta}$
and fast component $Y^{\e,\delta}$.  The functions $f, g,\sigma_1$  and $\sigma_2$ satisfy some appropriate conditions,
$\{W_t\}_{t\geq 0}$ is a standard  cylindrical Wiener process on the  Hilbert space $L^2(0,1)$,  and  $Q_1$, $Q_2$ are both trace class operators.

\vskip0.2cm
The motivation for the study of  multi-scale processes can be founded, for example,  in stochastic mechanics (see Freidlin and Wentzell \cite{FW98, FW2008}), where a polar change (or an appropriate change linked to the considered Hamiltonian) may give an amplitude evolving slowly whereas the phase is on an accelerated time scale; or in climate models (see Kiefer  \cite{Kief}), where climate-weather interactions may be studied within an averaging framework, climate being the slow motion and weather the fast one; or in genetic switching models (see Ge et al. \cite{GQX}),   which involves fast switching of DNA states between active and inactive states and the transcriptional and translational processes with different rates depending on the DNA states.

\vskip0.2cm
The study of the averaging principle has been extensively developed in both the deterministic (i.e., $\sigma_1=0$) and the stochastic context: see, for example,
Bogoliubov and Mitropolsky \cite{BM}
for the deterministic case;
 Khasminskii \cite{K1} for  a finite dimesional stochastic system; Cerrai and Freidlin \cite{CF} for an infinite dimensional stochastic reaction-diffusion systems. For more interesting results on this topic, we refer the reader to the recent works  \cite{CF,C1,L,B1,WR, DRW, WRDH, WR2013, FLL,CL,DSXZ, LRSX1}.

\vskip0.2cm

There have been a large body of researchers working on large deviation problems of multi-scale diffusions, just to list a few but far from being complete: Freidlin and Wentzell \cite[Chapter 7]{FW98},  Liptser \cite{Lip}, Veretennikov \cite{V0}  and Puhalskii \cite{Pu}.
Several well known large deviation approaches have been applied in this direction, see \cite{Ku,DS,HSS} for Dupuis-Ellis' weak convergence method \cite{DE}, and \cite{KP} for the viscosity method developed by Feng et al. \cite{FFK}. As a special multi-scale stochastic system,
slow-fast dynamics and its large deviation have also been studied in \cite{WRD,Sp,LL,HSS}  and the references therein.


 \vspace{2mm}
The novelty of our research is to apply a well known weak convergence criterion developed in \cite{Budhiraja-Dupuis} to prove the LDP of the slow-fast system \eref{Equation}. Thanks to this criterion, under some appropriate condition, we can simplify the delicate analysis in \cite{WRD,Sp,LL,HSS} to be an asymptotic behaviour study
of the following controlled process $(X^{\e,\delta, u^\e},Y^{\e,\delta, u^\e})$:
\begin{equation}\left\{\begin{array}{l}\label{R equation 1}
\frac{\partial }{\partial t}X^{\e,\delta, u^{\e}}_t(\xi)= \frac{\partial^2}{\partial \xi^2}X^{\e,\delta, u^{\e}}_t(\xi)+\frac{1}{2}\frac{\partial}{\partial \xi}\left(X^{\e,\delta, u^{\e}}_t\right)^2(\xi)+f\left(X^{\e,\delta, u^{\e}}_t, Y^{\e,\delta,, u^{\e}}_t\right)(\xi),\\
\vspace{2mm}
\ \ \ \ \ \ \ \ \ \ \ \ \ \ \  +\sigma_1\left(X^{\e,\delta, u^{\e}}_t\right)Q^{1/2}_1u^{\e}_t(\xi)+\sqrt{\e}\sigma_1\left(X_t^{\e,\delta, u^{\e}}\right)(\xi)Q_1^{1/2}\frac{\partial W}{\partial t}(t,\xi), \\
\vspace{2mm}
\frac{\partial }{\partial t}Y^{\e,\delta, u^{\e}}_t(\xi)
=\frac{1}{\delta}\big[\frac{\partial^2}{\partial \xi^2} Y^{\e,\delta, u^{\e}}_t(\xi)+g\left(X^{\e,\delta, u^{\e}}_t, Y^{\e,\delta, u^{\e}}_t\right)(\xi)\big]+\frac{1}{\sqrt{\delta\e }}\sigma_2\left(X^{\e,\delta, u^{\e}}_t,Y^{\e,\delta, u^{\e}}_t\right)Q^{1/2}_2u^{\e}_t(\xi)\\
\vspace{2mm}
\ \ \ \ \ \ \ \ \ \ \ \ \ \ \  +\frac{1}{\sqrt{\delta}}\sigma_2\left(X_t^{\e,\delta,u^{\e}}, Y_t^{\e,\delta, u^{\e}}\right)(\xi)Q_2^{1/2}\frac{\partial W }{\partial t}(t,\xi), \ (t,\xi)\in [0,T]\times [0,1]\\
\vspace{2mm}
X^{\e,\delta, u^{\e}}_t(0)=X^{\e,\delta, u^{\e}}_t(1)=Y^{\e,\delta, u^{\e}}_t(0)
=Y^{\e,\delta, u^{\e}}_t(1)=0, \quad t\in[0, T], \\
\vspace{2mm}
X^{\e, \delta, u^{\e}}_0=x, \quad Y^{\e,\delta, u^{\e}}_0=y, \end{array}\right.
\end{equation}
where $u^{\varepsilon}$ is a square integrable process often called control in sequel.

\vspace{2mm}
Another finding of this paper is that, as $\delta/\e\rightarrow 0$ and $\e\rightarrow 0$, it is surprising to see that the controlled slow processes $X^{\e,\delta, u^\e}$ would converge to a specific process $\bar{X}^u$ (see skeleton equation \eqref{eq sk} below).  This leads us to use the criterion in  \cite[Theorem 4.4]{Budhiraja-Dupuis} directly to get our large deviation result. We also need to stress that when $\delta/\e\nrightarrow 0$ as $\e\rightarrow 0$, the problem turns to be much more complicated and this straightforward method does not work anymore. For more applications of weak convergence criterion, we refer the reader to \cite{ZhZh15,WXZZ16,YaZh16,ZhZhZh18}.

\vspace{2mm}
In order to study the convergence of the controlled slow processes $X^{\e,\delta, u^\e}$, we will use the classical Khasminkii's time discretization approach, which is widely used in the proof of averaging principle. One difficulty here is from the non-linear term $\frac{1}{2}\frac{\partial}{\partial \xi}\left(X^{\e,\delta}_t\right)^2(\xi)$ in the Burgers equation, when proving the convergence of $X^{\e,\delta, u^{\e}}_t$, we have to use a stopping time technique developed in \cite{DSXZ,CSS} to handle this non-linearity. Moreover, to estimate the controlled system, we need to introduce an auxiliary process $\left(\hat{X}^{\e, \delta}, \hat{Y}^{\e, \delta}\right)$ (see Eq. \eqref{AuxiliaryPro Y 01} and Eq. \eqref{AuxiliaryPro X 01} below), which plays a crucial role for obtaining the weak convergence of controlled slow process. We believe that the method presented in this paper will be useful for studying the LDP for other types of
slow-fast stochastic partial differential equations.

\vskip0.2cm
The paper is organized as follows. In the next section, we introduce some notations
and assumptions used throughout the paper, then give the main result and  the  outline of the proof. Section 3 is devoted to the study of  averaged equation and skeleton equation. In  section 4,   we  finish the proof of  LDP  by the weak convergence approach. In  Appendix,  we recall some well-known results about the  LDP and  the Burgers equation.

\vspace{2mm}
In this paper, $C$ and $C_{p_1,\cdots, p_k}$  denote positive constants which may change from line to line along this paper, where   $C_{p_1,\cdots, p_k}$ is used to emphasize  that constant  depends on $p_1, \cdots, p_k, k\ge1$.

\section{Notations and main results} \label{Sec Main Result}

Let $\HH: = L^2(0,1)$ be the space of square integrable real-valued functions on   $[0,1]$.
The norm and the inner product on $\HH$ are
denoted by $|\cdot|$ and $\langle\cdot,\cdot\rangle$, respectively. For positive integer $k$,  let $\HH^k(0,1)$ be the Sobolev space of all functions in $\HH$ whose all derivatives up to the order $k$ also belong to $\HH$.  $\HH^1_0(0,1)$ is the subspace of $\HH^1(0,1)$ of all functions whose values at $0$ and $1$ vanish.

Let $A$  be the Laplace operator on $\HH$:
\begin{align*}
Ax:= \frac{\partial^2}{\partial \xi^2}x(\xi),
\quad  x\in D(A)=\HH^2(0,1) \cap \HH^1_0(0,1).
\end{align*}
It is well known that $A$ is the infinitesimal generator of a strongly continuous semigroup
$\{e^{tA}\}_{t\geq0}$. Let $\{e_k(\xi):=\sqrt{2}\sin(k\pi\xi)\}_{k\geq 1}$ be an orthonormal basis of $\HH$ consisting of the eigenvectors of $A$, i.e.,
$$A e_k=-\lambda_k e_k\ \ \ \ {\rm with} \ \ \lambda_k=k^2\pi^2.$$

For any $\sigma\in\RR$,  let $\HH_{\sigma}$ be the domain of the fractional operator $(-A)^{\sigma/2}$, i.e.,
$$\HH_{\sigma}:=D\left(\left(-A\right)^{\sigma/2}\right)=\left\{ x=\sum_{k \ge1} x_k e_k: (x_k)_{k\ge 1}\in \RR, \sum_{k\ge1} \lambda_k^{\sigma} |x_k|^2<\infty\right\},$$
with norm
$$
\ \|x\|_{\sigma}:=\left(\sum_{k\ge1} \lambda_k^{\sigma} |x_k|^2\right)^{1/2}.
$$
Then, for any $\sigma>0$,  $\HH_{\sigma}$ is densely and compactly embedded in $\HH$. Particularly,  $\VV:=\HH_{1}=\HH^1_0(0, 1)$, whose dual space is $\VV^{-1}$. The norm and the inner product on $\VV$ are
denoted by $\|\cdot\|$ and $\langle \cdot, \cdot\rangle_{\VV}$, respectively.

Define the bilinear operator $B(x,y): \HH \times \VV\rightarrow \VV^{-1}$ by
$$ B(x,y):=x\cdot\partial_{\xi} y,$$
and the trilinear operator $ b(x,y,z): \HH \times \VV\times \HH \rightarrow \RR$ by
$$ b(x,y,z)
:=\int_0^1x(\xi) \partial_\xi y(\xi)z(\xi) d\xi.$$
For convenience, set $B(x):=B(x,x)$, for $x\in \VV$.  The related properties about operators $e^{tA}$,  $b$ and $B$ are listed in Section 5.

With the above notations, the system \eqref{Equation} can be rewritten as:
\begin{equation}\left\{\begin{array}{l}\label{main equation}
\displaystyle
\vspace{2mm}
dX^{\e,\delta}_t=\left[AX^{\e,\delta}_t+B\left(X^{\e,\delta}_t\right)+f\left(X^{\e,\delta}_t, Y^{\e,\delta}_t\right)\right]dt+\sqrt{\e}\sigma_1\left(X^{\e,\delta}_t\right)Q^{1/2}_1dW_t,\\
\vspace{2mm}
dY^{\e,\delta}_t=\frac{1}{\delta}\left[AY^{\e,\delta}_t+g\left(X^{\e,\delta}_t, Y^{\e,\delta}_t\right)\right]dt+\frac{1}{\sqrt\delta}\sigma_2\left(X^{\e,\delta}_t,Y^{\e,\delta}_t\right)Q^{1/2}_2dW_t,\\
  X^{\e,\delta}_0=x, \quad Y^{\e,\delta}_0=y.\end{array}\right.
\end{equation}
Here, $W$ denotes a standard cylindrical Wiener process  on $\HH$. Since $Q_1$ and  $Q_2$ are trace class operators, the embedding of $Q_i^{1/2}\HH$ in $\HH$ is Hilbert-Schmidt for $i=1,2$. Let $\mathcal L_2(\HH;\HH)$ denote the space of   Hilbert-Schmidt operators from $\HH$ to $\HH$, endowed with the Hilbert-Schmidt norm $\|G\|_{\HS}=\sqrt{\tr(GG^*)}=\sqrt{\sum_k |Ge_k|^2}$.

Suppose that   $f, g: \HH\times \HH \rightarrow \HH$, $\sigma_1Q^{1/2}_1: \HH \rightarrow \mathcal L_2(\HH; \HH)$, $\sigma_2Q^{1/2}_2: \HH\times \HH \rightarrow \mathcal L_2(\HH; \HH)$ satisfy the following conditions:
\begin{conditionA}\label{A1}
$f$, $g$, $\sigma_1$ and $\sigma_2$ are Lipschitz continuous, i.e., there exist some positive constants $L_{f}, L_{g}, L_{\sigma_1}, L_{\sigma_2}$ and $C>0$ such that for any $x_1,x_2,y_1,y_2\in \HH$,
\begin{align*}
&|f(x_1, y_1)-f(x_2, y_2)|\leq L_{f}\left(|x_1-x_2| + |y_1-y_2|\right);\\
&|g(x_1, y_1)-g(x_2, y_2)|\leq C|x_1-x_2| + L_{g}|y_1-y_2|;\\
&\left\|[\sigma_1(x_1)-\sigma_1(x_2)]Q^{1/2}_1\right\|_{\HS}\leq L_{\sigma_1}|x_1-x_2|;\\
&\left\|[\sigma_2(x_1, y_1)-\sigma_2(x_2, y_2)]Q^{1/2}_2\right\|_{\HS} \leq C|x_1-x_2| + L_{\sigma_2}|y_1-y_2|.
\end{align*}
\end{conditionA}

\begin{conditionA}\label{A2}
There exists   $C>0$ such that for any $x\in\HH$,
$$
\sup_{y\in\HH}\left\|\sigma_2(x,y)Q^{1/2}_2\right\|_{\HS} \leq C(|x|+1).
$$
\end{conditionA}

\begin{conditionA}\label{A3} The smallest eigenvalue $\lambda_1$ of $-\Delta$ and the Lipschitz constants $L_g$, $L_{\sigma_2}$ satisfy
$$
\lambda_{1}-L_{g}>0 \quad \text{and }\quad \frac{L^2_{\sigma_2}}{\lambda_1}+\frac{L_g}{\lambda_1-L_g}<1.
$$
\end{conditionA}

\begin{conditionA}\label{A4}  $\lim_{\e\downarrow 0}\de(\e)=0$ and $\lim_{\e\downarrow 0}\frac{\de}{\e}=0.$
\end{conditionA}

\begin{remark}  Condition \ref{A3} is not a sharp condition and can be weakened by more accurate calculus.  Condition \ref{A4} is a very important condition to make sure that the additional controlled term in $Y^{\e,\delta, u^\e}$ converges to $0$ (see Remark \ref{Rem 4.6} for details).
\end{remark}

Following the standard approach developed in \cite{DPZ}, one can prove that under
Condition \ref{A1}, there exists a unique mild solution to the system \eqref{main equation}. More specifically, for any given initial value $x, y\in \HH$ and $T>0$, there exists a unique solution $(X^{\varepsilon,\delta}, Y^{\varepsilon,\delta})$ such that
$X^{\varepsilon,\delta}, Y^{\varepsilon,\delta} \in C([0,T]; \HH) \cap L^2(0, T; \VV)$
satisfying that
\begin{equation}\left\{\begin{array}{l}\label{mild solution}
\displaystyle
\vspace{2mm}
X^{\varepsilon,\delta}_t=e^{tA}x+\int^t_0e^{(t-s)A}B\left(X^{\varepsilon,\delta}_s\right)ds+\int^t_0e^{(t-s)A}f\left(X^{\varepsilon,\delta}_s, Y^{\varepsilon,\delta}_s\right)ds\\
\vspace{2mm}
\ \ \ \ \ \ \ \ \ \ \ \ \ +\sqrt{\e}\int^t_0 e^{(t-s)A}\sigma_1\left(X^{\e,\delta}_s\right)Q^{1/2}_1dW_s,\\
\vspace{2mm}
Y^{\varepsilon,\delta}_t=e^{tA/\delta}y+\frac{1}{\delta}\int^t_0e^{(t-s)A/\delta}g\left(X^{\varepsilon,\delta}_s,Y^{\varepsilon,\delta}_s\right)ds
+\frac{1}{\sqrt{\delta}}\int^t_0 e^{(t-s)A/\delta}\sigma_2\left(X^{\e,\delta}_s,Y_2^{\e,\delta}\right)Q^{1/2}_2dW_s,\\
 X^{\e, \delta}_0=x, \quad Y^{\e,\delta}_0=y.
 \end{array}\right.
\end{equation}
Let   $\Gamma^{\e}$ be the  functional  from $C([0,T]; \HH)$ into $C([0,T]; \HH) \cap L^2(0, T; \VV)$ satisfying that
\begin{equation}\label{eq solu function}
\Gamma^{\e}(W):=X^{\varepsilon,\delta}.
\end{equation}

Consider the following skeleton equation:
 \begin{equation}\label{eq sk}
\left\{\begin{array}{l}
\displaystyle d \bar{X}^u_{t}=\left[A \bar{X}^u_{t}+B\left(\bar{X}^u_{t}\right)+\bar{f}\left(\bar{X}^u_{t}\right)\right]dt+\sigma_1\left(\bar{X}^u_t\right)Q_1^{1/2}u(t) dt,\\
\bar{X}^u_{0}=x,\end{array}\right.
\end{equation}
where $u\in L^2([0,T];\HH)$ and
\begin{eqnarray*}
\bar{f}(x)=\int_{\HH}f(x,y)\mu^{x}(dy), \quad x\in \HH,
\end{eqnarray*}
with $\mu^{x}(\cdot)$ being the unique invariant measure of the transition semigroup for the corresponding frozen equation (see Eq. \eref{FEQ} below).
By Lemma \ref{barX} below,  Eq. \eqref{eq sk} adimts a unique solution, and  we denote the solution as follows
 \begin{equation}\label{eq Gamma}
 \Gamma^0\left(\int_0^{\cdot} u(s)ds\right):= \bar{X}^u.
 \end{equation}

\vspace{3mm}
 The main result of this paper is the following theorem.
\begin{theorem}\label{main result 1}
Under  \ref{A1}-\ref{A4}, $\{X^{\e, \de}\}_{\e>0}$ satisfies the LDP in $C([0, T]; \HH)$ with the rate function $I$ given by
$$
I(g):=\inf_{\left\{u\in L^2([0,T]; \HH);g=\Gamma^0\left(\int_0^{\cdot}u(s)ds\right)\right\}}\left\{\frac12\int_0^T|u(s)|^2ds\right\},\ g\in C([0, T]; \HH).
$$
\end{theorem}
\noindent \textbf{Proof of Theorem \ref{main result 1}:}
According to the weak convergence criteria in Theorem \ref{thm BD}, we  just need to prove that two conditions (a) and (b) in Theorem \ref{thm BD} are fulfilled. Condition
(b)  will be established in Proposition \ref{Prop Gamm 0 compact} in the following section, and the verification of Condition (a)  will be given by Propositions \ref{convergence 1} and  \ref{convergence 2} in Section 4.


\section{The frozen  equation and skeleton equation}

In this section,  we will  prove  Condition (b) in Theorem \ref*{thm BD} to prove the LDP.  Before proving the compactness of solutions $\{\bar{X}^u\}$ to the skeleton equation \eqref{eq sk}, a  frozen equation is also introduced.  The unique invariant measure of the frozen equation is applied to define the coefficient $\bar{f}$ in the skeleton equation,  and the Lipschitz continuity of $\bar{f}$ is used a lot in the following discussion. Note that we assume conditions \ref{A1}-\ref{A3} hold in this section.

\subsection{The frozen and skeleton equations}

For any fixed $x\in \HH$, we first consider the following frozen equation
associated with the fast component:
\begin{equation} \label{FEQ}
dY_{t}=AY_{t}dt+g\left(x,Y_{t}\right)dt+\sigma_2 (x,Y_{t})Q_2^{1/2}d  \widetilde{W}_t,\quad Y_{0}=y,
\end{equation}
where $\widetilde{W}_t$ is a standard  cylindrical  Wiener process independent of $W_t$. Since $g(x,\cdot)$ and $\sigma_2(x,\cdot)Q_2^{1/2}$ are Lipshcitz continuous,
it is easy to prove that for any fixed $y \in \HH$,
 Eq. $\eref{FEQ}$ has a unique mild solution denoted by $Y_{t}^{x,y}$.
Moreover,  $Y^{x,y}_t$ is a homogeneous Markov process, and  let $P^{x}_t$ be the transition semigroup of $Y_{t}^{x,y}$,
that is, for any bounded measurable function $\varphi$ on $\HH$,
\begin{align*}
P^x_t \varphi(y):= \mathbb{E} \left[\varphi\left(Y_{t}^{x,y}\right)\right], \quad y \in \HH,\ \ t>0.
\end{align*}
Under Condition \ref{A3},
it is easy to prove that $\sup_{t\geq 0}\EE\left[\left|Y_{t}^{x,y}\right|^2\right]\leq C(1+|x|^2+|y|^2)$ and $P^x_t$ has unique invariant measure
$\mu^x$. We here give the following asymptotic behavior of $P^x_t$ proved in \cite{C1}.
\begin{proposition} \cite[(2.13)]{C1}\label{ergodicity}
There exist $C, \eta>0$ satisfying that for any Lipschitz continuous function $\varphi: \HH \to \mathbb{R}$,
$$
\left| P^{x}_t\varphi(y)-\int_{\HH}\varphi(z)\mu^x(dz)\right|
\leq C\left(1+ |x| + |y|\right)e^{-\eta t}\|\varphi\|_{Lip}, \ \ \ \forall x,y\in \HH, t>0,
$$
where $\|\varphi\|_{Lip}:=\sup_{x,y\in \HH, x \neq y}\frac{|\varphi(x)-\varphi(y)|}{|x-y|}$.
\end{proposition}

\begin{lemma} \label{L3.17} There exists a constant $C>0$ satisfying that for any $x_1, x_2, y\in \HH$,
\begin{eqnarray*}
\sup_{t\geq 0}\EE\left|Y^{x_1,y}_t-Y^{x_2,y}_t\right|^2\leq C|x_1-x_2|^2.
\end{eqnarray*}
\end{lemma}
\begin{proof}
Note that $Y^{x_1,y}_0-Y^{x_2,y}_0=0$ and
\begin{align*}
d(Y^{x_1,y}_t-Y^{x_2,y}_t)=&A(Y^{x_1,y}_t-Y^{x_2,y}_t)dt+\left[g\left(x_1, Y^{x_1,y}_t\right)-g\left(x_2, Y^{x_2,y}_t\right)\right]dt\\
&+\left[\sigma_2\left(x_1, Y^{x_1,y}_t\right)-\sigma_2\left(x_2, Y^{x_2,y}_t\right)\right]Q_2^{1/2}d\widetilde W_t.
\end{align*}

By It\^{o}'s formula and Condition \ref{A1}, we get
\begin{eqnarray*}
\frac{d}{dt}\EE|Y^{x_1,y}_t-Y^{x_2,y}_t|^2=\!\!\!\!\!\!\!\!&&-2\EE\|Y^{x_1,y}_t-Y^{x_2,y}_t\|^2+2\EE\langle g\left(x_1, Y^{x_1,y}_t\right)-g\left(x_2, Y^{x_2,y}_t\right), Y^{x_1,y}_t-Y^{x_2,y}_t\rangle \\
\!\!\!\!\!\!\!\!&&+\EE\|\left[\sigma_2\left(x_1, Y^{x_1,y}_t\right)-\sigma_2\left(x_2, Y^{x_2,y}_t\right)\right]Q_2^{1/2}\|^2_{\HS}\\
\leq\!\!\!\!\!\!\!\!&&-2\lambda_1\EE|Y^{x_1,y}_t-Y^{x_2,y}_t|^2+2L_g\EE|Y^{x_1,y}_t-Y^{x_2,y}_t|^2+2C|x_1-x_2|\cdot \EE|Y^{x_1,y}_t-Y^{x_2,y}_t|\\
\!\!\!\!\!\!\!\!&&+\EE \left(C|x_1-x_2|+L_{\sigma_2}|Y^{x_1,y}_t-Y^{x_2,y}_t|\right)^2.
\end{eqnarray*}

By Condition \ref{A3}, it is easy to see that $2\lambda_1-2L_g-L^2_{\sigma_2}>0$. Then by Young's inequality, there exists a constant $\gamma>0$ such that
$$
\frac{d}{dt}\EE|Y^{x_1,y}_t-Y^{x_2,y}_t|^2\leq-\gamma\EE|Y^{x_1,y}_s-Y^{x_2,y}_s|^2+C|x_1-x_2|^2.
$$

Hence, the comparsion theorem implies for any $t>0$,
\begin{eqnarray*}
\EE\left|Y^{x_1,y}_t-Y^{x_2,y}_t\right|^2\leq C|x_1-x_2|^2\int^t_0 e^{-\gamma(t-s)}ds\leq C|x_1-x_2|^2/\gamma.
\end{eqnarray*}
The proof is complete.
\end{proof}

\vspace{0.3cm}

Let $\mathcal{K}$ be a Hilbert space endowed with norm $\|\cdot\|_{\mathcal{K}}$.  For $p>1$, $\alpha\in(0,1)$, let $W^{\alpha,p}([0,T];\mathcal{K})$ be the
Sobolev space of all $u\in L^p([0,T]; \mathcal{K})$ such that
$$
\int_0^T\int_0^T\frac{\|u(t)-u(s)\|^p_\mathcal{K}}{|t-s|^{1+\alpha p}}\;dtds<\infty,
$$
endowed with the norm
$$
\|u\|^p_{W^{\alpha,p}([0,T]; \mathcal{K})}:=\int_0^T\|u(t)\|^p_{\mathcal{K}}\;dt+\int_0^T\int_0^T\frac{\|u(t)-u(s)\|^p_\mathcal{K}}{|t-s|^{1+\alpha p}}\;dtds.
$$

The following result represents a variant of the criteria for compactness proved in \cite[Sect. 5, Ch. I]{Lions},
 and \cite[Sect. 13.3]{Temam 1983} .
\begin{lemma}\label{Compact}{\rm
Let $\mathbb{S}_0\subset \mathbb{S}\subset \mathbb{S}_1$ be Banach spaces, $\mathbb{S}_0$ and $\mathbb{S}_1$ reflexive, with compact embedding of $\mathbb{S}_0$ in $\mathbb{S}$.
For $p\in(1,\infty)$ and $\alpha\in(0,1)$, let $\Lambda$ be the space
$$
\Lambda=L^p([0,T];\mathbb{S}_0)\cap W^{\alpha,p}([0,T];\mathbb{S}_1)
$$
endowed with the natural norm. Then the embedding of $\Lambda$ in $L^p([0,T];\mathbb{S})$ is compact.
}\end{lemma}

Let $\mathbb S=L^2([0,T]; \HH)$, and
$\mathcal A$ denotes  the class of  $\{\FF_t\}$-predictable processes  taking values in $\HH$ a.s..
Let $\mathbb S_N=\{u\in  \mathbb S; \int_0^T|u(s)|^2ds\le N\}$. The set $\mathbb S_N$ endowed with the weak topology is a Polish space.
Define $\mathcal A_N=\{u\in \mathcal A;u(\omega)\in \mathbb S_N, \mathbb{P}\text{-a.s.}\}$.

\vspace{2mm}
Recall $\bar X^u$ given in the skeleton equation \eqref{eq sk}.  The existence and uniqueness of the solution  of Eq. \eqref{eq sk} is given in the following lemma.
\begin{lemma} \label{barX}
For any $x\in \HH$, $u\in \mathbb S$, Eq. \eqref{eq sk} admits a unique mild solution $\bar{X}^u \in C([0,T]; \HH)\cap L^2([0,T]; \VV)$.  Moreover,
 for any $N>0$ and $\alpha\in(0,1/2)$, there exist constants $C_{N, T}$ and $C_{\alpha, N, T}$ such that
\begin{equation}\label{eq skeleton estimate}
\sup_{u\in \mathbb S_N}\left\{\sup_{t\in [0, T]}\left|\bar X^u_t\right|^2+\int_0^T\left\|\bar X^u_s\right\|^2ds\right\}\le C_{N, T}(1+|x|^2),
\end{equation}
and
\begin{equation}\label{eq sobolev}
\sup_{u\in \mathbb S_N}\left\|\bar X^u\right\|_{W^{\alpha, 2}([0,T];\VV^{-1})}\le C_{\alpha, N, T}(1+|x|^2).
\end{equation}

\end{lemma}
\begin{proof}
\textbf{Step1.  (Existence and uniqueness of the solution): } If $\bar f$ is Lipschitz continuous,  the existence and uniqueness of the solution can be proved
similarly as in the case of the Burgers equation.

In fact, for any $x_1,x_2,y\in \HH$ and $t>0$, by Proposition \ref{ergodicity} and Lemma \ref{L3.17}, we have
\begin{align*}
&\left|\bar{f}(x_1)-\bar{f}(x_2)\right|\\
\leq&\left|\int_{\HH} f(x_1,z)\mu^{x_1}(dz)-\int_{\HH} f(x_2,z)\mu^{x_2}(dz)\right|\\
\leq&\left|\int_{\HH} f(x_1,z)\mu^{x_1}(dz)-\EE\left[ f\left(x_1, Y^{x_1,y}_t\right)\right]\right|+\left|\EE\left[ f\left(x_2, Y^{x_2,y}_t\right)\right]-\int_{\HH} f(x_2,z)\mu^{x_2}(dz)\right|\\
 &+\left|\EE\left[ f\left(x_1, Y^{x_1,y}_t\right)\right]-\EE\left[ f\left(x_2, Y^{x_2,y}_t\right)\right]\right|\\
\leq&C\left(1+|x_1|+|x_2|+|y|\right)e^{-\eta t}+C\left(|x_1-x_2|+\EE\left|Y^{x_1,y}_t-Y^{x_2,y}_t\right|\right)\\
\leq &C\left(1+|x_1|+|x_2|+|y|\right)e^{-\eta t}+C|x_1-x_2|.
\end{align*}
Letting $t\rightarrow \infty$, we have
\begin{equation}\label{eq Lip}
\left|\bar{f}(x_1)-\bar{f}(x_2)\right|\le C|x_1-x_2|.
\end{equation}

\textbf{Step 2. (Proof of \eqref{eq skeleton estimate})}:
 For any $u\in \mathbb S_N$, by \ref{A1}, \eqref{eq Lip} and Lemma \ref{Property B0},   we have
\begin{align*}
&\left|\bar X^u_t\right|^2+2\int_0^t\left\|\bar X^u_s\right\|^2ds\\
=&|x|^2+2\int_0^t \left\langle B\left(\bar X^u_s\right), \bar X^u_s\right\rangle ds+2\int_0^t \left\langle \bar f\left(\bar X^u_s\right), \bar X^u_s\right\rangle ds+2\int_0^t\left\langle \sigma_1\left(\bar X^u_s\right)Q^{1/2}_1 u(s), \bar X^u_s\right\rangle ds\\
\le&|x|^2+C\int_0^t \left(1+\left|\bar X^u_s\right|^2\right) ds+2\int_0^t \left|u(s)\right|\cdot \left\|\sigma_1\left(\bar X^u_s\right)Q_1^{1/2}\right\|_{\HS}\cdot\left|\bar X^u_s\right|ds\\
\le &|x|^2+C\int_0^t\left(1+\left|\bar X^u_s\right|^2 \right)ds+\int_0^t\left\|\sigma_1\left(\bar X^u_s\right)Q_1^{1/2}\right\|_{\HS}^2ds+\int_0^t\left|u(s)\right|^2\cdot\left|\bar X^u_s\right|^2ds\\
\le &|x|^2+  C  \int_0^t\left|\bar X^u_s\right|^2\left(1+ \left|u(s)\right|^2\right)ds+Ct.
\end{align*}
Since $u\in \SS_N$, by Gronwall's inequality,  we get
\begin{align*}
\sup_{t\in[0,T]}\left|\bar X^u_t\right|^2+\int_0^T\left\|\bar X^u_s\right\|^2ds\le&  C_T\left(1+ \left|x\right|^2\right)\exp\left\{\int_0^T C\left(1+|u(s)|^2\right) ds \right\}\\
\le&  \left(1+|x|^2\right) C_{N,T}<\infty,
\end{align*}
which yields  \eqref{eq skeleton estimate}.

\vspace{2mm}
\textbf{ Step 3. (Proof of \eqref{eq sobolev}):}   Notice that
\begin{align}
\bar X^u_t&=x+\int_0^tA \bar X^u_sds+\int_0^tB\left(\bar X^u_s\right)ds+\int_0^t\bar f\left(\bar X^u_s\right)ds+\int_0^t\sigma_1\left(\bar X^u_s\right)Q^{1/2}_1 u(s)ds\notag\\
&=:x+I_1(t)+I_2(t)+I_3(t)+I_4(t).
\end{align}
Using the same arguments as that in  the proof of \cite[Theorem 3.1]{Flandoli-Gatarek}, we have
\begin{equation}\label{eq Sob 1}
\|I_1\|^2_{W^{\alpha,2}([0,T];\VV^{-1})}\le C.
\end{equation}
By Corollary \ref{Property B3} and the Cauchy-Schwartz inequality,   for any $0\le s\le t\le T$,
\begin{align*}
\left\|\int_s^t B\left(\bar X^u_r\right)dr\right\|_{-1}^2\le &\left(\int_s^t\left\|B\left(\bar X^u_r\right)\right\|_{-1}\;dr\right)^2
\le   C\left(\int_s^t \left|\bar X^u_r\right|\cdot \left \|\bar X^u_r\right\|   dr\right)^2 \notag\\
\le &C\left(\int_0^T\left|\bar X^u_r\right|^2dr\right). \left(\int_s^t\left\|\bar X^u_r\right\|^2dr\right).
\end{align*}
Thus,
\begin{align}\label{Sob 21}
\int_0^T\|I_2(s)\|^2_{-1}\; ds\le& C_T \left(\int_0^T\left|\bar X^u_r\right|^2dr\right)\left(\int_0^T\left\|\bar{X}^u_r\right\|^2dr\right) <+\infty,
\end{align}
and
\begin{align}\label{Sob 22}
\int_0^T\int_0^T\frac{\|I_2(t)-I_2(s)\|^2_{-1}}{|t-s|^{1+2\alpha}}dsdt\le& C_T\left(\int_0^T\left|\bar X^u_r\right|^2dr\right)
\times \int_0^T\int_0^T\int_s^t\frac{\left\| \bar X^u_r\right\|^2}{|t-s|^{1+2\alpha}}drdsdt.
\end{align}
By the Cauchy-Schwartz inequality and Fubini's theorem, there exists $C_{\alpha,T}>0$ such that
\begin{align}\label{Sob 23}
\int_0^T\int_0^T\int_s^t\frac{\left\|\bar X^u_r\right\|^2}{|t-s|^{1+2\alpha}}drdsdt \le  C_{\alpha,T} \int_0^T\left\|\bar X^u_r\right\|^2dr.
\end{align}
Combining \eqref{eq skeleton estimate}, \eqref{Sob 21}, \eqref{Sob 22} and \eqref{Sob 23}, we have
\begin{align}\label{eq Sob 2}
\|I_2\|^2_{W^{\alpha,2}([0,T]; \VV^{-1})}\le C_{\alpha, T}(1+|x|^2).
\end{align}

Similarly, we also have
\begin{align}\label{eq Sob 3}
\|I_3\|^2_{W^{\alpha,2}([0,T];{\VV^{-1}})}\le C_{\alpha, T}(1+|x|^2).
\end{align}

It remains to deal with the last term $I_4$. Since $u\in \mathbb S_N$,  by \ref{A2}, we have
\begin{align*}
\int_0^T\left\|\int_0^t\sigma_1\left(\bar X^u_s\right)Q_1^{1/2}   u(s)ds\right\|_{-1}^2\;dt
&\le C\int_0^T \left(\int_0^t \left\|\sigma_1\left(\bar X^u_s\right)Q_1^{1/2}\right\|_{\HS} \cdot\left|u(s)\right|ds\right)^2dt\\
&\le  C_T \int_0^Tc\left(1+ \left| \bar X^u_s\right|^2\right)ds \cdot\int_0^T\left| u(s)\right|^2ds\\
&\le C_{N, T}
\end{align*}
and
\begin{align*}
\left\|\int_s^t\sigma_1\left( \bar X_r^u\right)Q^{1/2}_1  u(r)dr\right\|^2_{-1}&\le C\int_s^t\left\|\sigma_1\left(\bar X^u_r\right)Q_1^{1/2}\right\|_{\HS}^2dr\cdot\int_s^t\left|u(r)\right|^2dr\\
&\le C_{N}\int_s^t\left(1+\left|\bar X^u_r\right|^2\right)dr.
\end{align*}
Similar to (\ref{eq Sob 2}), the above two inequalities imply that
\begin{equation}\label{eq Sob 4}
\|I_4\|^2_{W^{\alpha,2}([0,T];\VV^{-1})}\le C_{\alpha, N, T}(1+|x|^2).
\end{equation}
By \eqref{eq Sob 1}, \eqref{eq Sob 2}, \eqref{eq Sob 3} and \eqref{eq Sob 4}, we obtain \eqref{eq sobolev}.
The proof is complete.
\end{proof}

\subsection{Compactness of solutions to skeleton equations}

Recall that for $u\in \mathbb S$,  $\bar X^u$ is the solution of  the skeleton equation  \eqref{eq sk} and
\begin{equation} \label{solu skel}\Gamma^0\left(\int_0^\cdot u(s)ds\right)=\bar X^u.
\end{equation}

\begin{proposition}\label{Prop Gamm 0 compact} For  any $N<\infty$,
  the family
  $$\mathbb K_N:= \left\{\Gamma^0\left(\int_0^{\cdot} u(s)ds\right); u\in \mathbb S_N\right\}$$
  is compact in $C([0,T]; \HH)\cap L^2([0,T]; \VV)$.
 \end{proposition}

\begin{proof}  Choose a sequence $\{u_n\in\SS_N; n\geq 1\}$,  and let $\left\{\bar X^{u_n}=\Gamma^0(\int_0^{\cdot}  u_n(s)ds);n\ge1\right\}$ be a  sequence of elements in $C([0,T]; \HH)\cap L^2([0,T]; \VV)$.
Lemma \ref{Compact},  together with   \eqref{eq skeleton estimate} and \eqref{eq sobolev}, enables us to assert that there exist a subsequence $\{n'\}$ and $u\in \mathbb S_N$  such that
\begin{center}
\begin{enumerate}
 \item[(a)] $u_{n'}\rightarrow u$ in $\mathbb S_N$ weakly, as $n'\rightarrow \infty$ ;
 \vspace{2mm}
 \item[(b)] $\bar X^{u_{n'}}\rightarrow \bar X^u$ in $L^2([0,T];\HH)$ strongly;
 \vspace{2mm}
  \item[(c)] $\sup_{n'\ge1} \sup_{0\le t\le T}\left|\bar X^{u_{n'}}(t)\right|<\infty$.
\end{enumerate}
\end{center}

Using the same argument as in the proof of  \cite[Theorem 3.1]{Temam}, we can conclude that
 $\bar X^u=\Gamma^0(\int_0^{\cdot}  u(s)ds)$.
Next, we will prove that $\bar X^{u_{n'}}\rightarrow \bar X^u$ in $C([0,T]; \HH)\cap L^2([0,T]; \VV)$.

 Using  Lemma \ref{Property B2}, we obtain
 \begin{align}\label{eq bar X}
&\left|\bar X^{u_{n'}}_t-\bar X^u_t\right|^2+2\int_0^t\left\|\bar X^{u_{n'}}_s-\bar X^u_s\right\|^2ds\notag\\
=&2\int_0^t\left\langle B\left(\bar X^{u_{n'}}_s\right)-B\left(\bar X^u_s\right), \bar X^{u_{n'}}_s-\bar X^u_s\right\rangle ds\notag\\
&+2\int_0^t\left\langle \bar f\left(\bar X^{u_{n'}}_s\right)-\bar f\left(\bar X^u_s\right), \bar X^{u_{n'}}_s-\bar X^u_s\right\rangle ds\notag\\
&+2\int_0^t\left\langle \sigma_1\left(\bar X^{u_{n'}}_s\right)Q^{1/2}_1\left[u_{n'}(s)- u(s)\right], \bar X^{u_{n'}}_s-\bar X^u_s \right\rangle ds\notag\\
&+2\int_0^t\left\langle \left[\sigma_1\left(\bar X^{u_{n'}}_s\right) - \sigma_1\left(\bar X^{u_n}_s\right)\right]Q_1^{1/2}u(s), \bar X^{u_{n'}}_s-\bar X^u_s \right\rangle ds\notag\\
=:&I_1^n(t)+I_2^n(t)+I_3^n(t)+I_4^n(t).
\end{align}

For the first term, by the elementary inequality $2ab\le  a^2+b^2$ for  $a, b>0$, we have
\begin{align}\label{eq bar X1}
|I_1^n(t)|\le &2c\int_0^t \left|\bar X^{u_{n'}}_s-\bar X^u_s\right|\cdot\left\|\bar X^{u_{n'}}_s-\bar X^u_s\right\|\cdot\left(\left\|\bar X^u_s\right\|+\left\|\bar X^{u_{n'}}_s\right\|\right)ds\notag\\
\le&  \int_0^T  \left\|\bar X^{u_{n'}}_s-\bar X^u_s\right\|^2 ds+  C\int_0^t \left|\bar X^{u_{n'}}_s-\bar X^u_s\right|^2\cdot\left(\left\|\bar X^u_s\right\|^2+\left\|\bar X^{u_{n'}}_s\right\|^2\right)ds.
\end{align}

For the  second term,  by  the Lipschitz continuity of $\bar{f}$ and (b), we have
\begin{align}\label{eq bar X2}
\sup_{t\in[0, T]}|I_2^n(t)|\le
  C\int_0^T\left|\bar X^{u_{n'}}_s-\bar X^u_s\right|^2ds\rightarrow 0,
\end{align}

For the third term,  by the linear growth condition of $\sigma_1Q_1^{1/2}$, (b) and (c),
\begin{align}\label{eq bar X3}
&\sup_{t\in[0, T]}\left|I_3^n(t)\right|\le  2\int_0^T\left\|\sigma_1\left(\bar X^{u_{n'}}_s\right)Q_1^{1/2}\right\|_{\HS} \cdot\left|u_{n'}(s)-u(s)\right|\cdot\left|\bar X^{u_{n'}}_s-\bar X^u_s\right|ds\notag\\
\le & C\left(1+\sup_{0\le s\le T}\left|\bar X_s^{u_{n'}}\right|\right)\left(\int_0^T  \left|u_{n'}(s)-u(s)\right|^2 ds\right)^{1/2}\cdot\left(\int_0^T\left| \bar X^{u_{n'}}_s-\bar X^u_s\right|^2ds\right)^{1/2}\notag\\
\le & 2CN\left(1+\sup_{0\le s\le T}\left|\bar X_s^{u_{n'}}\right|\right) \cdot\left(\int_0^T\left| \bar X^{u_{n'}}_s-\bar X^u_s\right|^2ds\right)^{1/2}\notag\\
&\longrightarrow 0.
\end{align}

For the last term, by Condition \ref{A1}, we have
\begin{align}\label{eq bar X4}
|I_4^n(t)|\le
C\int_0^t \left |u(s)\right|\cdot\left|\bar X^{u_{n'}}_s-\bar X^u_s\right|^2ds.
\end{align}

By \eqref{eq bar X}-\eqref{eq bar X4}, we have
  \begin{align*}
&\sup_{s\in[0, t]}\left|\bar X^{u_{n'}}_s-\bar X^u_s\right|^2+\int_0^t\left\|\bar X^{u_{n'}}_s-\bar X^u_s\right\|^2ds\\
\le &C\int_0^t\left(\left\|\bar X_s^u\right\|^2+\left\|\bar X_s^{u_{n'}}\right\|^2+|u(s)|\right)\left|\bar X^{u_{n'}}_s-\bar X^u_s\right|^2ds+\sup_{0\le s\le t}\left(I_2^n(s)+I_3^n(s)\right).
\end{align*}
By Gronwall's inequality and (a)-(c), we have
 \begin{align*}
\sup_{t\in[0, T]}\left|\bar X^{u_{n'}}_t-\bar X^u_t\right|^2+\int_0^T\left\|\bar X^{u_{n'}}_s-\bar X^u_s\right\|^2ds\rightarrow 0,\ \ \mbox{as}\ \ n'\rightarrow \infty.
\end{align*}
This implies that  $\mathbb K_N$ is compact in $C([0,T];\HH)\cap L^2([0,T];\VV)$.
The proof is complete.
\end{proof}

\section{Convergence of the controlled  slow processes}

In this section we will finish the proof of main result by verifying   Condition (a) in Theorem \ref*{thm BD}. Before that, a series of auxiliary results are needed to prove the convergence of the process $X^{\e, \delta, u^{\e}}$.  Note that we assume Conditions \ref{A1}-\ref{A4} hold in this section.
 \subsection{The auxiliary controlled equation}

For every fixed  $N\in\mathbb{N},\e>0, \delta>0$,    let $u^\e\in \mathcal{A}_N$ and $\Gamma^{\e}$ be given by \eqref{eq solu function}.
By Girsanov's theorem, we know that
 $$X^{\e,\delta, u^{\e}}:=\Gamma^\e\left(W(\cdot)+\frac1{\sqrt{\e}}\int_0^{\cdot}u^\e(s)ds\right)$$
   is a part of the solution
$\left(X^{\e,\delta, u^{\e}}, Y^{\e, \delta, u^{\e}}\right)$ of the following controlled equation:
\begin{equation}\left\{\begin{array}{l}\label{R equation}
\displaystyle
\vspace{2mm}
dX^{\e,\delta, u^{\e}}_t=\left[AX^{\e,\delta, u^{\e}}_t+B\left(X^{\e,\delta, u^{\e}}_t\right)+f\left(X^{\e,\delta, u^{\e}}_t, Y^{\e,\delta, u^{\e}}_t\right)\right]dt+\sigma_1\left(X^{\e,\delta, u^{\e}}_t\right)Q^{1/2}_1u^{\e}(t) dt\\
\vspace{2mm}
\quad\quad\quad\quad\quad+\sqrt{\e}\sigma_1\left(X^{\e,\delta, u^{\e}}_t\right)Q^{1/2}_1dW_t,\\
\vspace{2mm}
dY^{\e,\delta, u^{\e}}_t=\frac{1}{\delta}\left[AY^{\e,\delta,u^{\e}}_t+g\left(X^{\e,\delta, u^{\e}}_t, Y^{\e,\delta, u^{\e}}_t\right)\right]dt+\frac{1}{\sqrt{\delta\e }}\sigma_2\left(X^{\e,\delta, u^{\e}}_t,Y^{\e,\delta, u^{\e}}_t\right)Q^{1/2}_2u^{\e}(t)dt\\
\vspace{2mm}
\quad\quad\quad\quad\quad+\frac{1}{\sqrt{\delta}}\sigma_2\left(X^{\e,\delta, u^{\e}}_t,Y^{\e,\delta, u^{\e}}_t\right)Q^{1/2}_2dW_t,\\
\vspace{2mm}
X^{\e,\delta,u^{\e}}_0=x,\quad\quad Y^{\e,\delta,u^{\e}}_0=y.\end{array}\right.
\end{equation}

We first prove the uniform boundedness of the solutions $\left(X^{\e, \de, u^{\e}}, Y^{\e, \de, u^{\e}}\right)$ to the system \eref{R equation} for all $\varepsilon,\de \in (0,1)$.

\begin{lemma} \label{PE}
For any $x,y\in \HH$, $T>0$ and $\{u^\e;\e>0\}\subset \mathcal{A}_N$, there exists a constant $C_{T}>0$ such that for all $\e,\de\in(0,1)$,
\begin{align} \label{Control X}
\mathbb{E}\left(\sup_{t\in[0, T]}\left|X_{t}^{\e,\de,u^{\e}}\right|^{2}\right)+\EE\int^T_0 \left\|X_{t}^{\e,\de,u^{\e}} \right\|^2dt
\leq  C_{T}\left(1+|x|^{2}+|y|^2\right)
\end{align}
and
\begin{align} \label{Control Y}
\mathbb{E} \int^T_0\left|Y_{t}^{\e,\de,u^{\e}} \right|^{2}dt
\leq C_{T}\left(1+ |x|^{2} + |y|^{2}\right).
\end{align}
\end{lemma}

\begin{proof}
According to It\^{o}'s formula, we have
\begin{eqnarray} \label{ItoFormu 1}
 \mathbb{E} \left[\left|Y_{t}^{\e,\de, u^{\e}} \right|^{2}\right]=  \!\!\!\!\!\!\!\!&&|y|^2
-\frac{2}{\de}\mathbb{E} \int_0^t \left\|Y_{s}^{\e,\de, u^{\e}}\right\|^2 ds +\frac{2}{\de}\mathbb{E}\left|\int_0^t\left\langle g\left(X_{s}^{\e,\de, u^{\e}},Y_{s}^{\e,\de, u^{\e}}\right),Y_{s}^{\e,\de, u^{\e}}\right\rangle  ds\right|\nonumber\\
&&+ \frac{2}{\sqrt{\e \de}}\mathbb{E}\left|\int_0^t\left\langle\sigma_2\left(X_{s}^{\e,\de, u^{\e}},Y_{s}^{\e,\de, u^{\e}}\right)Q_2^{1/2}u^{\e}(s),Y_{s}^{\e,\de, u^{\e}}\right\rangle ds\right|\nonumber\\
&&+\frac{1}{\de}\mathbb{E} \int_0^t\left\| \sigma_2\left(X_{s}^{\e,\de, u^{\e}},Y_{s}^{\e,\de, u^{\e}}\right) Q_2^{1/2}\right\|^{2}_{\HS}ds.
\end{eqnarray}

By Poincar\'e's inequality and  \ref{A1} and \ref{A2}, it follows from \eqref{ItoFormu 1} that
\begin{eqnarray*}
\frac{d}{dt}\mathbb{E}\left[\left| Y_{t}^{\e,\de, u^{\e}} \right|^{2}\right]\leq \!\!\!\!\!\!\!\!&&
-\frac{2\lambda_{1}}{\de} \mathbb{E}\left[\left|Y_{t}^{\e,\de, u^{\e}}\right|^{2}\right]\!+\frac{2}{\de}\mathbb{E}\left(C\left |Y_{t}^{\e,\de, u^{\e}}\right |+C \left| X_{t}^{\e,\de, u^{\e}} \right| \cdot \left|Y_{t}^{\e,\de, u^{\e}} \right|+L_g\left|Y_{t}^{\e,\de, u^{\e}} \right|^2\right)\nonumber \\
 \!\!\!\!\!\!\!\!&&+\frac{C L_{\sigma_2}}{\sqrt{\e \de}}\mathbb{E}\left[\left(1+\left|X_{t}^{\e,\de, u^{\e}}\right|\right)\left|u^{\e}(t)\right|\cdot\left|Y_{t}^{\e,\de, u^{\e}}\right|\right]+\frac{C L_{\sigma}^2}{\de}\mathbb{E} \left(1+\left|X_{t}^{\e,\de, u^{\e}}\right|^2\right).
\end{eqnarray*}
Using \ref{A3} and Young's inequality, we deduce that
\begin{eqnarray*}
\frac{d}{dt}\mathbb{E} \left[\left|Y_{t}^{\e,\de, u^{\e}} \right|^{2}\right]
\leq \!\!\!\!\!\!\!\!&& -\frac{\lambda_1-L_g}{\de}\mathbb{E}\left[\left |Y_{t}^{\e,\de, u^{\e}} \right|^{2}\right]+\frac{C}{\de}\mathbb{E}\left(\left|X_{t}^{\e,\de, u^{\e}} \right|^{2}+1\right)\\
&&+\frac{C}{\sqrt{\e \de}}\EE\left[\left(1+\left| X_{t}^{\e,\de, u^{\e}} \right|^{2}\right)\left|u^{\e}(t)\right|^2\right].
\end{eqnarray*}
By the comparison theorem, we have
\begin{eqnarray*}
\mathbb{E}\left[\left| Y_{t}^{\e,\de,u^{\e}} \right|^{2}\right]\leq \!\!\!\!\!\!\!\!&&|y|^{2} e^{-\frac{\lambda_1-L_g}{\de}t}+\frac{C}{\de}\int^t_0 e^{-\frac{\lambda_1-L_g}{\de}(t-s)}\left(\mathbb{E}\left|X_{s}^{\e,\de,u^{\e}}\right|^{2}+1\right)ds\\
&&+\frac{C}{\sqrt{\e \de}}\EE\int^t_0e^{-\frac{\lambda_1-L_g}{\de}(t-s)}   \left(1+\left| X_{s}^{\e,\de, u^{\e}} \right|^{2}\right)\left|u^{\e}(s)\right|^2ds.
\end{eqnarray*}
Then we have
\begin{eqnarray}
\mathbb{E}\int^T_0 \left| Y_{t}^{\e,\de,u^{\e}} \right|^{2}dt
\leq \!\!\!\!\!\!\!\!&&|y|^{2} \int^T_0e^{-\frac{\lambda_1-L_g}{\de}t}dt+\frac{C}{\de}\int^T_0\int^t_0 e^{-\frac{\lambda_1-L_g}{\de}(t-s)}\left[\mathbb{E}\left|X_{s}^{\e,\de,u^{\e}}\right|^{2}+1\right]dsdt\nonumber\\
&&+\frac{C}{\sqrt{\e \de}}\EE\left\{\left(1+\sup_{s\in [0, T]}\left| X_{s}^{\e,\de, u^{\e}}\right |^{2}\right)\int^T_0\int^t_0 e^{-\frac{\lambda_1-L_g}{\de}(t-s)} \left|u^{\e}(s)\right|^2dsdt\right\}\nonumber\\
\leq \!\!\!\!\!\!\!\!&&
C\left(1+|y|^{2}\right)+C\int^T_0 \mathbb{E}\left|X_{t}^{\e,\de,u^{\e}} \right|^{2}dt+\frac{C_N\sqrt{\de}}{\sqrt{\e }} \EE \left[\sup_{s\in [0,T]}\left|X_{s}^{\e,\de, u^{\e}} \right|^{2}\right].\label{EY}
\end{eqnarray}

 Applying It\^{o}'s formula, we have
\begin{eqnarray*}
\left| X_{t}^{\e,\de,u^{\e}} \right|^{2}= \!\!\!\!\!\!\!\!&& |x|^2-
\int^t_0 2\left\|X_{s}^{\e,\de,u^{\e}}\right\|^2 ds+2\int_0^t \left\langle B\left(X_{s}^{\e,\de,u^{\e}}\right), X_{s}^{\e,\de,u^{\e}}\right\rangle\;ds\nonumber \\
\!\!\!\!\!\!\!\!&&+2\int^t_0\left\langle f\left(X_{s}^{\e,\de,u^{\e}},Y_{s}^{\e,\de,u^{\e}}\right),X_{s}^{\e,\de,u^{\e}}\right\rangle ds+ 2\sqrt{\e}\int^t_0 \left\langle X_{s}^{\e,\de,u^{\e}}, \sigma_1\left(X_{s}^{\e,\de,u^{\e}}\right)Q_1^{1/2}dW_s\right\rangle\nonumber\\
\!\!\!\!\!\!\!\!&&+2\int^t_0 \left\langle X_{s}^{\e,\de,u^{\e}},  \sigma_1\left(X_{s}^{\e,\de,u^{\e}}\right)Q_{1}^{1/2}u^{\e}(s)\right\rangle ds+\e\int^t_0\left\|\sigma_1\left( X_{s}^{\e,\de,u^{\e}}\right)Q_1^{1/2}\right\|_{\HS}^2ds.
\end{eqnarray*}

By Lemma \ref{Property B0},    \ref{A1} and \ref{A2},  we obtain that
\begin{eqnarray*}
&&\left| X_{t}^{\e,\de,u^{\e}} \right|^{2}+\int^t_0 \left\|X_{s}^{\e,\de,u^{\e}}\right\|^2 ds\\
\leq \!\!\!\!\!\!\!\!&& C+|x|^2+C\int^t_0\left |X_{s}^{\e,\de,u^{\e}}\right|^2 ds+C\int^t_0 \left| Y_{s}^{\e,\de,u^{\e}}\right |^2ds+2\sqrt{\e}\int^t_0 \left\langle  X_{s}^{\e,\de,u^{\e}}, \sigma_1\left(X_{s}^{\e,\de,u^{\e}}\right)Q_1^{1/2}dW_s \right\rangle\\
&&+C\int^t_0 \left(1+\left|X_{s}^{\e,\de,u^{\e}}\right|\right) \left|u^{\e}(s)\right|\cdot \left|X_{s}^{\e,\de,u^{\e}}\right|ds+\e C\int^t_0\left(1+\left|X_{s}^{\e,\de,u^{\e}}\right|^2\right)ds\\
\leq \!\!\!\!\!\!\!\!&& C+|x|^2 +C\int^t_0\left|X_{s}^{\e,\de,u^{\e}}\right|^2 ds+C\int^t_0 \left| Y_{s}^{\e,\de,u^{\e}} \right|^2ds+2\sqrt{\e}\int^t_0 \left\langle
 X_{s}^{\e,\de,u^{\e}},  \sigma_1\left(X_{s}^{\e,\de,u^{\e}}\right)Q^{1/2}_1dW_s\right\rangle\\
&&+\frac{1}{4}\sup_{s\in[0, t]}\left| X_{s}^{\e,\de,u^{\e}} \right|^{2}.
\end{eqnarray*}
Then by Burkholder-Davis-Gundy's inequality and \eref{EY}, we have
\begin{eqnarray*}
&&\EE\left(\sup_{t\in[0, T]}\left| X_{t}^{\e,\de,u^{\e}} \right|^{2}\right)+\EE\int^T_0 \left\|X_{s}^{\e,\de,u^{\e}}\right\|^2 ds\\
\leq\!\!\!\!\!\!\!\!&& C\left(1+|x|^2\right)+C\EE\int^T_0 \left|X_{s}^{\e,\de,u^{\e}}\right|^2 ds+C\EE\int^T_0 \left| Y_{s}^{\e,\de,u^{\e}}\right|^2ds\\
&&+C\sqrt{\e}\EE\left[\sup_{t\in[0, T]}\left|\int^t_0 \left\langle X_{s}^{\e,\de,u^{\e}}, \sigma_1\left(X_{s}^{\e,\de,u^{\e}}\right)Q_1^{1/2}dW_s\right\rangle\right|\right]\\
\leq\!\!\!\!\!\!\!\!&& C\left(1+|x|^2+|y|^2\right)+C\int^T_0\left|X_{s}^{\e,\de,u^{\e}}\right|^2 ds+\frac{C_N\sqrt{\de}}{\sqrt{\e }}\EE\left(\sup_{t\in[0, T]}
\left| X_{t}^{\e,\de,u^{\e}} \right|^{2}\right)\\
&&+C\EE\left[\int^T_0 \left(1+\left|X_{s}^{\e,\de,u^{\e}}\right|^2\right)\cdot\left|X_{s}^{\e,\de,u^{\e}}\right|^2ds\right]^{1/2}\\
\leq\!\!\!\!\!\!\!\!&& C\left(1+|x|^2+|y|^2\right)+C\int^T_0 \EE\left[\left|X_{s}^{\e,\de,u^{\e}}\right|^2\right] ds+\left(\frac{1}{4}+\frac{C_N\sqrt{\de}}{\sqrt{\e }}\right)\EE\left(\sup_{t\in[0, T]}\left| X_{t}^{\e,\de,u^{\e}}\right|^{2}\right).\end{eqnarray*}
By \ref{A4}, taking $\e$ small enough such that $\de/\e\le \frac14$ we have,
\begin{eqnarray*}
\EE\left[\sup_{t\in[0, T]}\left| X_{t}^{\e,\de,u^{\e}} \right|^{2}\right]+\EE\int^T_0 \left\|X_{s}^{\e,\de,u^{\e}}\right\|^2 ds
\leq \!\!\!\!\!\!\!\!&& C\left(1+|x|^2+|y|^2\right)+C\EE\int^T_0\left|X_{s}^{\e,\de,u^{\e}}\right|^2 ds.
\end{eqnarray*}
By Gronwall's inequality, we have
\begin{eqnarray}
\EE\left[\sup_{t\in[0, T]}\left| X_{t}^{\e,\de,u^{\e}} \right|^{2}\right]+\EE\int^T_0\left\|X_{s}^{\e,\de,u^{\e}}\right\|^2 ds
\leq \!\!\!\!\!\!\!\!&& C_T\left(1+|x|^2+|y|^2\right).\label{EX}
\end{eqnarray}
The inequality \eqref{Control Y} follows by combining \eqref{EY} and \eqref{EX}.
The proof is complete.
\end{proof}

Because the approach based on the time discretization will be used later, we need the following lemma, which is inspired from \cite[Lemma 3.2]{LSXZ} and plays an important role in the proof.
Meanwhile, it will be very helpful to weaken the regularity  requirement of initial value $x$, i.e., we drop the regularity of initial value $x\in \HH_{\theta}$ with $\theta\in (1, 3/2)$ in \cite{DSXZ} and only assume $x\in \HH$ here. To this purpose, we first construct the following stopping time, for any $R,\e >0$,
\begin{equation}
 \tau^{\e}_R:=\inf\left\{t>0, \left|X^{\e,\de,u^{\e}}_t\right|>R\right\}.\nonumber
\end{equation}

\begin{lemma} \label{COX}
For any $x, y\in\HH$, $R, T>0$ and $\e, \Delta>0$ small enough,  there exists a constant $C_{R, T}>0$ such that
\begin{align}
\mathbb{E}\left[\int^{T\wedge \tau^{\e}_R}_0|X^{\e,\de,u^{\e}}_t-X^{\e,\de,u^{\e}}_{t(\Delta)}|^2 dt\right]\leq C_{R,T}\Delta^{1/2}(1+|x|^2+|y|^2),\label{F3.7}
\end{align}
where $t(\Delta):=[\frac{t}{\Delta}]\Delta$ and $[s]$ denotes the largest integer which is not greater  than $s$.
\end{lemma}

\begin{proof}
By a straightforward compute,
\begin{eqnarray}
&&\mathbb{E}\left[\int^{T\wedge \tau^{\e}_R}_0\left|X^{\e,\de,u^{\e}}_t-X^{\e,\de,u^{\e}}_{t(\Delta)}\right|^2 dt\right]\nonumber\\
\leq\!\!\!\!\!\!\!\!&& \mathbb{E}\left(\int^{\Delta}_0\left|X^{\e,\de,u^{\e}}_t-x\right|^2 1_{\{t\leq \tau^{\e}_R\}}dt\right)+\mathbb{E}\left[\int^{T}_{\Delta}\left|X^{\e,\de,u^{\e}}_t-X^{\e,\de,u^{\e}}_{t(\Delta)}\right|^2 1_{\{t\leq \tau^{\e}_R\}}dt\right]\nonumber\\
\leq\!\!\!\!\!\!\!\!&& C_R\left(1+|x|^2\right)\Delta +2\mathbb{E}\left(\int^{T}_{\Delta}\left|X^{\e,\de,u^{\e}}_{t}-X^{\e,\de,u^{\e}}_{t-\Delta}\right|^2 1_{\{t\leq \tau^{\e}_R\}}dt\right)\nonumber\\
&&+2\mathbb{E}\left(\int^{T}_{\Delta}\left|X^{\e,\de,u^{\e}}_{t(\Delta)}-X^{\e,\de,u^{\e}}_{t-\Delta}\right|^2 1_{\{t\leq \tau^{\e}_R\}}dt\right).\label{F3.8}
\end{eqnarray}
Firstly, we estimate the second term on the right-hand side of \eref{F3.8}.   Applying  It\^{o}'s formula to $Z_u:=X^{\e,\de,u^{\e}}_u-X^{\e,\de,u^{\e}}_{t-\Delta}$, we have the increment of $|Z_u|^2$ over interval $[t-\Delta, t]$  as follows,
\begin{eqnarray}
&&\left|X^{\e,\de,u^{\e}}_t-X^{\e,\de,u^{\e}}_{t-\Delta}\right|^{2}\nonumber\\
=\!\!\!\!\!\!\!\!&&2\int_{t-\Delta}^{t}\left\langle A X^{\e,\de,u^{\e}}_s\!\!+B\left(X^{\e,\de,u^{\e}}_s\right), X^{\e,\de,u^{\e}}_s-X^{\e,\de,u^{\e}}_{t-\Delta}\right\rangle ds\nonumber\\
\!\!\!\!\!\!\!\!&&+2\int_{t-\Delta} ^{t}\left\langle f\left(X^{\e,\de,u^{\e}}_s, Y^{\e,\de,u^{\e}}_s\right), X^{\e,\de,u^{\e}}_s-X^{\e,\de,u^{\e}}_{t-\Delta}\right\rangle ds\nonumber \\
\!\!\!\!\!\!\!\!&&+2\int_{t-\Delta} ^{t}\left\langle \sigma_1\left(X^{\e,\de,u^{\e}}_s\right)Q^{1/2}_1 u^{\e}_s, X^{\e,\de,u^{\e}}_s-X^{\e,\de,u^{\e}}_{t-\Delta}\right\rangle ds+\e\int_{t-\Delta} ^{t}\left\|\sigma_1\left(X^{\e,\de,u^{\e}}_s\right)Q^{1/2}_1\right\|_{\HS}^2ds\nonumber\\
\!\!\!\!\!\!\!\!&&+2\sqrt{\e}\int_{t-\Delta} ^{t}\left\langle X^{\e,\de,u^{\e}}_s-X^{\e,\de,u^{\e}}_{t-\Delta},  \sigma_1\left(X^{\e,\de,u^{\e}}_s\right)Q^{1/2}_1dW_s\right\rangle \nonumber\\
:=\!\!\!\!\!\!\!\!&&L_{1}(t)+L_{2}(t)+L_{3}(t)+L_{4}(t)+L_{5}(t).  \label{F3.9}
\end{eqnarray}

For the term $L_1(t)$, by H\"{o}lder's inequality, Corollary \ref{Property B3} and the definition of stopping time $\tau^{\e}_R$, we have
\begin{eqnarray}  \label{REGX1}
&&\mathbb{E}\left(\int^{T}_{\Delta}|L_{1}(t)|1_{\{t\leq \tau^{\e}_R\}}dt\right)\nonumber\\
\leq\!\!\!\!\!\!\!\!&&C\mathbb{E}\left(\int^{T}_{\Delta}\int_{t-\Delta} ^{t}\left\| A X^{\e,\de,u^{\e}}_s+B\left(X^{\e,\de,u^{\e}}_s\right)\right\|_{-1}
\left\|X^{\e,\de,u^{\e}}_s-X^{\e,\de,u^{\e}}_{t-\Delta}\right\| ds 1_{\{t\leq \tau^{\e}_R\}}dt\right)\nonumber\\
\leq\!\!\!\!\!\!\!\!&&C\left[\mathbb{E}\int^{T}_{\Delta}\int_{t-\Delta} ^{t}\!\!\left\|AX^{\e,\de,u^{\e}}_s-B\left(X^{\e,\de,u^{\e}}_s\right)\right\|^2_{-1}ds1_{\{t\leq \tau^{\e}_R\}}dt\right]^{1/2}\!\!
\left[\mathbb{E}\int^{T}_{\Delta}\int_{t-\Delta} ^{t}\!\!\left\|X^{\e,\de,u^{\e}}_s-X^{\e,\de,u^{\e}}_{t-\Delta}\right\|^2 ds 1_{\{t\leq \tau^{\e}_R\}}dt\right]^{1/2}\nonumber\\
\leq\!\!\!\!\!\!\!\!&&C\left[\Delta\mathbb{E}\int^{T}_0\left\|X^{\e,\de,u^{\e}}_s\right\|^2\left(1+\left|X^{\e,\de,u^{\e}}_s\right|^2\right)1_{\{s\leq \tau^{\e}_R\}}ds\right]^{1/2}\cdot\left[\Delta\mathbb{E}\int^{T}_0\left\|X^{\e,\de,u^{\e}}_s\right\|^2ds\right]^{1/2}\nonumber\\
\leq\!\!\!\!\!\!\!\!&&C_{R,T}\Delta(1+|x|^2+|y|^2),
\end{eqnarray}
where we use the Fubini theorem and \eref{Control X} in the third and fourth inequalities respectively.

For the term $L_{2}(t)$, by Condition \ref{A1} and \eref{Control Y}, we get
\begin{eqnarray}\label{REGX2}
&&\mathbb{E}\left(\int^{T}_{\Delta}|L_{2}(t)|1_{\{t\leq \tau^{\e}_R\}}dt\right)\nonumber\\
\leq\!\!\!\!\!\!\!\!&&C\mathbb{E}\left(\int^{T}_{\Delta}\int_{t-\Delta} ^{t}\left(1+\left|X^{\e,\de,u^{\e}}_s\right|+\left|Y^{\e,\de,u^{\e}}_s\right|\right)\left(\left|X^{\e,\de,u^{\e}}_s\right|+\left|X^{\e,\de,u^{\e}}_{t-\Delta}\right|\right)ds 1_{\{t\leq \tau^{\e}_R\}} dt\right)\nonumber\\
\leq\!\!\!\!\!\!\!\!&&C_{R,T}\Delta+C_{R}\mathbb{E}\int^T_{\Delta}\int^t_{t-\Delta}\left|Y^{\e,\de,u^{\e}}_s\right|dsdt\nonumber\\
\leq\!\!\!\!\!\!\!\!&&C_{R,T}\Delta+C_{R,T}\Delta \left[\mathbb{E}\int^T_{0}\left|Y^{\e,\de,u^{\e}}_s\right|^2ds\right]^{1/2}\nonumber\\
\leq\!\!\!\!\!\!\!\!&&C_{R,T}\Delta(1+|x|^2+|y|^2).
\end{eqnarray}

For the terms $L_{3}(t)$ and $L_{4}(t)$, it easy to see
\begin{eqnarray}\label{REGX2a}
&&\mathbb{E}\left(\int^{T}_{\Delta}|L_{3}(t)|1_{\{t\leq \tau^{\e}_R\}}dt\right)\nonumber\\
\leq\!\!\!\!\!\!\!\!&&C\mathbb{E}\left(\int^{T}_{\Delta}\int_{t-\Delta} ^{t}\left(1+\left|X^{\e,\de,u^{\e}}_s\right|\right )|u^{\e}_s|\left(\left|X^{\e,\de,u^{\e}}_s\right|+\left|X^{\e,\de,u^{\e}}_{t-\Delta}\right|\right)ds 1_{\{t\leq \tau^{\e}_R\}} dt\right)\nonumber\\
\leq\!\!\!\!\!\!\!\!&&C_{R,T}\Delta+C_{R}\mathbb{E}\int^T_{\Delta}\int^t_{s-\Delta}\left|u^{\e}_s\right|dsdt\nonumber\\
\leq\!\!\!\!\!\!\!\!&&C_{R,T}\Delta+C_{R,T}\Delta \left[\mathbb{E}\int^T_{0}\left|u^{\e}_s\right|^2ds\right]^{1/2}\nonumber\\
\leq\!\!\!\!\!\!\!\!&&C_{R,T}\Delta,
\end{eqnarray}
and
\begin{eqnarray}\label{REGX2b}
\mathbb{E}\left(\int^{T}_{\Delta}|L_{4}(t)|1_{\{t\leq \tau^{\e}_R\}}dt\right)\leq\!\!\!\!\!\!\!\!&&C\mathbb{E}\left(\int^{T}_{\Delta}\int_{t-\Delta} ^{t}\left(1+\left|X^{\e,\de,u^{\e}}_s\right|^2\right)1_{\{s\leq \tau^{\e}_R\}}ds dt\right)\nonumber\\
\leq\!\!\!\!\!\!\!\!&&C_{R, T}\Delta.
\end{eqnarray}

For the term $L_{5}(t)$, Burkholder-Davies-Gundy's inequality implies
\begin{eqnarray}  \label{REGX3}
&&\mathbb{E}\left(\int^{T}_{\Delta}|L_{5}(t)|1_{\{t\leq \tau^{\e}_R\}}dt\right)\nonumber\\
\leq\!\!\!\!\!\!\!\!&&C\mathbb{E}\int^{T}_{\Delta}\left[\int_{t-\Delta} ^{t}\left\|\sigma_1(X^{\e,\de,u^{\e}}_s)Q^{1/2}_1\right\|^2_{\HS}\left|X^{\e,\de,u^{\e}}_s-X^{\e,\de,u^{\e}}_{t-\Delta}\right|^2 1_{\{s\leq \tau^{\e}_R\}}ds\right]^{1/2}dt\nonumber\\
\leq\!\!\!\!\!\!\!\!&&C_T\left[\mathbb{E}\int^{T}_{\Delta}\int^t_{t-\Delta}\left(1+\left|X^{\e,\de,u^{\e}}_s\right|^2\right)\left(\left|X^{\e,\de,u^{\e}}_s\right|^2+\left|X^{\e,\de,u^{\e}}_{t-\Delta}\right|^2\right)1_{\{s\leq \tau^{\e}_R\}}dsdt\right]^{1/2}\nonumber\\
\leq\!\!\!\!\!\!\!\!&&C_{R,T}\Delta^{1/2}.
\end{eqnarray}
Combining estimates \eref{F3.9}-\eref{REGX3} together, we can deduce that
\begin{eqnarray}
\mathbb{E}\left(\int^{T}_{\Delta}\left|X^{\e,\de,u^{\e}}_t-X^{\e,\de,u^{\e}}_{t-\Delta}\right|^2 1_{\{t\leq \tau^{\e}_R\}}dt\right)\leq\!\!\!\!\!\!\!\!&&C_{R,T}\Delta^{1/2}(1+|x|^2+|y|^2). \label{F3.13}
\end{eqnarray}
By the similar argument above, we can also get
\begin{eqnarray}
\mathbb{E}\left(\int^{T}_{\Delta}\left|X^{\e,\de,u^{\e}}_{t(\Delta)}-X^{\e,\de,u^{\e}}_{t-\Delta}\right|^2 1_{\{t\leq \tau^{\e}_R\}}dt\right)\leq\!\!\!\!\!\!\!\!&&C_{R,T}\Delta^{1/2}(1+|x|^2+|y|^2). \label{F3.14}
\end{eqnarray}
Hence, \eref{F3.8}, \eref{F3.13} and \eref{F3.14} implies \eref{F3.7} holds. The proof is complete.
\end{proof}

 \subsection{Some   priori estimates on auxiliary processes}

Following the idea inspired by Khasminskii \cite{K1},
we introduce an auxiliary process.
Specifically,
we split the interval $[0,T]$ into some subintervals of size $\Delta>0$, and we will let $\Delta=\delta^{1/2}$ later.
With the initial value $\hat{Y}_{0}^{\e,\delta}=Y^{\e}_{0}=y$,
we construct the process $\hat{Y}_{t}^{\e,\delta}$ as follows:
$$
d\hat{Y}_{t}^{\e,\delta}=\frac{1}{\delta}\left[A\hat{Y}_{t}^{\e,\delta}+g\left(X^{\e,\delta, u^{\e}}_{t(\Delta)},\hat{Y}_{t}^{\e,\de}\right)\right]dt+\frac{1}{\sqrt{\de}}\sigma_2\left(X^{\e,\delta, u^{\e}}_{t(\Delta)},\hat{Y}_{t}^{\e,\de}\right)Q_2^{1/2}dW_t,
$$
which satisfies
\begin{align} \label{AuxiliaryPro Y 01}
\hat{Y}_{t}^{\e,\delta}=&e^{tA/\de}y+\frac{1}{\de}\int_{0}^{t}e^{(t-s)A/\de}g\left(X_{s(\Delta)}^{\e,\delta,u^{\e}},\hat{Y}_{s}^{\e,\delta}\right)ds\notag\\
&+\frac{1}{\sqrt{\de}}\int_{0}^{t}e^{(t-s)A/\de}\sigma_2\left(X^{\e,\delta, u^{\e}}_{s(\Delta)},\hat{Y}_{s}^{\e,\de}\right)Q_2^{1/2}dW_s,
\end{align}
where $t(\Delta)=[\frac{t}{\Delta}]\Delta$ is the nearest breakpoint proceeding $t$. We construct the process $\hat{X}_{t}^{\e,\de}$
as follows:
$$
\frac{d}{dt}\hat{X}_{t}^{\e,\de}=A\hat{X}_{t}^{\e,\de}+B\left(\hat X_{t}^{\e,\de}\right)+f\left(X_{t(\Delta)}^{\e,\de,u^{\e}},\hat{Y}_{t}^{\e,\de}\right)+\sigma_1\left(\hat{X}_{t}^{\e,\de}\right)Q_1^{1/2}u^{\e}(t),\quad \hat{X}_{0}^{\e,\de}=x.
$$
Then
\begin{align} \label{AuxiliaryPro X 01}
\hat X^{\e,\de}_t=&e^{tA}x+\int^t_0e^{(t-s)A}B\left(\hat X^{\e,\de}_s\right)ds+\int^t_0e^{(t-s)A}f\left( X^{\e,\de, u^{\e}}_{s(\Delta)}, \hat Y^{\e,\de}_s\right)ds\notag\\
&+\int^t_0 e^{(t-s)A}\sigma_1\left(\hat{X}_{s}^{\e,\de}\right)Q_1^{1/2}u^{\e}(s)ds.
\end{align}

The following Lemma gives a control of the auxiliary processes $\left(\hat{X}_{t}^{\e,\de},\hat{Y}_{t}^{\e,\de}\right)$.
Since the proof can be carried out almost the same way as in the proof of
Lemma \ref{PE}, we omit the proof here.
\begin{lemma} \label{AE}
For any $x,y\in \HH$ and $T>0$, there exists  a constant $C_T>0$ such that for all $\e,\delta\in(0,1]$
\begin{align} \label{hatXHolderalpha}
\mathbb{E}\left(\sup_{t\in [0,T]}\left|\hat X_{t}^{\e,\de} \right|^{2}\right)+\EE\int^T_0\left\|\hat X_{t}^{\e,\de} \right\|^2dt\leq  C_T\left(1+ |x|^{2}+|y|^2\right)
\end{align}
and
\begin{align}
\sup_{t\in [0,T]} \mathbb{E}\left| \hat{Y}_{t}^{\e,\de}\right|^{2}\leq C_T\left(1+ |x |^{2}+ |y|^{2}\right).\label{hat Y}
\end{align}
\end{lemma}

\begin{lemma} \label{L3.1} For any $x,y\in \HH$, $R, T>0$, there exists a constant $C_{R, T}>0$ such that for all $\e,\delta\in(0,1]$
\begin{equation}\label{eq L3.1}
\EE\int^{T\wedge\tau^{\e}_R}_0 \left|Y^{\e,\delta, u^{\e}}_s-\hat{Y}^{\e,\delta}_s\right|^2 ds\leq C_{R, T}(1+|x|^2+|y|^2)\Delta^{1/2}+\frac{C_{R,T}\sqrt{\de}}{\sqrt{ \e}}.
\end{equation}
\end{lemma}
\begin{proof}
Let $\rho_t:=Y^{\e,\delta, u^{\e}}_t-\hat{Y}^{\e,\delta}_t$ and   $\Lambda_t:=\rho_t-V_t-M_t$, with
$$
V_t:=\frac{1}{\sqrt{\delta \e}}\int^t_0 e^{\frac{(t-s)A}{\delta}}\sigma_2\left(X^{\e,\delta,u^{\e}}_s,Y^{\e,\delta,u^{\e}}_s\right)Q_2^{1/2}u^{\e}(s)ds
$$
and
$$
M_t:=\frac{1}{\sqrt{\delta}}\int^t_0 e^{\frac{(t-s)A}{\delta}}\left[\sigma_2\left(X^{\e,\delta,u^{\e}}_s,Y^{\e,\delta,u^{\e}}_s\right)-\sigma_2\left(X^{\e,\delta,u^\e}_{s(\Delta)},\hat{Y}^{\e,\delta}_s\right)\right]Q_2^{1/2}dW_s.
$$
Then it is easy to see that $\Lambda_t$ satisfies the following equation:
\begin{eqnarray*}
d\Lambda_t=\frac{1}{\delta}\left[A\Lambda_t+g\left(X^{\e,\delta,u^{\e}}_t,Y^{\e,\delta,u^{\e}}_t\right)-g\left(X^{\e,\delta,u^\e}_{t(\Delta)}, \hat{Y}^{\e,\delta}_t\right)\right]dt,\quad \Lambda_0=0.
\end{eqnarray*}
Thus,
\begin{eqnarray*}
\frac{d}{dt}|\Lambda_t|^{2}=\!\!\!\!\!\!\!\!&&-\frac{2}{\delta}\|\Lambda_t \|^{2}+\frac{2}{\de}\left\langle g\left(X^{\e,\delta,u^{\e}}_t,Y^{\e,\delta,u^{\e}}_t\right)-g\left(X^{\e,\delta,u^\e}_{t(\Delta)}, \hat{Y}^{\e,\de}_t\right), \Lambda_t\right\rangle\\
\leq\!\!\!\!\!\!\!\!&&-\frac{2\lambda_1}{\delta}\left|\Lambda_t \right|^{2}+\frac{C}{\de}\left |X^{\e, \de, u^{\e}}_t-X^{\e,\delta,u^{\e}}_{t(\Delta)}\right|\cdot\left|\Lambda_t\right|+\frac{2L_g}{\de}\left|\rho_t\right|\cdot\left|\Lambda_t\right|\\
\leq\!\!\!\!\!\!\!\!&&-\frac{2\lambda_1}{\delta}\left|\Lambda_t \right|^{2}+\frac{C}{\de}\left |X^{\e, \de, u^{\e}}_t-X^{\e,\delta,u^{\e}}_{t(\Delta)}\right|^2+\frac{(\lambda_1-L_g)}{\delta}\left|\Lambda_t \right|^{2}+\frac{2L_g}{\de}\left|\Lambda_t\right|^2+\frac{L_g}{2\de}\left|\rho_t\right|^2\\
\leq\!\!\!\!\!\!\!\!&&-\frac{(\lambda_1-L_g)}{\delta}\left|\Lambda_t \right|^{2}+\frac{C}{\de} \left|X^{\e, \de, u^{\e}}_t-X^{\e,\delta,u^{\e}}_{t(\Delta)}\right|^2+\frac{L_g}{2\de}\left|\rho_t\right|^2.
\end{eqnarray*}
By the comparison theorem, we have
\begin{eqnarray*}
|\Lambda_t|^{2}\leq\!\!\!\!\!\!\!\!&&\frac{C}{\de}\int^t_0 e^{-\frac{(\lambda_1-L_g)(t-s)}{\delta}}\left|X^{\e, \de, u^{\e}}_s-X^{\e,\delta,u^\e}_{s(\Delta)}\right|^2ds+\frac{L_g}{2\delta}\int^t_0e^{-\frac{(\lambda_1-L_g) (t-s)}{\delta}}\left|\rho_s\right|^2ds.
\end{eqnarray*}
Then by Fubini's Theorem, for any $T>0$,
\begin{eqnarray*}
&&\int^T_0|\Lambda_t|^{2}dt\\
\leq\!\!\!\!\!\!\!\!&&\frac{C}{\de}\int^T_0\int^t_0 e^{-\frac{(\lambda_1-L_g) (t-s)}{\delta}}\left|X^{\e, \de, u^{\e}}_s-X^{\e,\delta,u^\e}_{s(\Delta)}\right|^2 dsdt+\frac{L_g}{2\de}\int^T_0\int^t_0 e^{-\frac{(\lambda_1-L_g) (t-s)}{\delta}}|\rho_s|^2 dsdt\\
=\!\!\!\!\!\!\!\!&&\frac{C}{\de}\int^T_0\left(\int^T_s e^{-\frac{(\lambda_1-L_g)(t-s)}{\delta}}dt\right)\left|X^{\e, \de, u^{\e}}_s-X^{\e,\delta,u^\e}_{s(\Delta)}\right|^2 ds+\frac{L_g}{2\de}\int^T_0\left(\int^T_s e^{-\frac{(\lambda_1-L_g)(t-s)}{\delta}}dt\right)|\rho_s|^2 ds\\
\leq\!\!\!\!\!\!\!\!&& C\int^T_0\left|X^{\e, \de, u^{\e}}_s-X^{\e,\delta,u^{\e}}_{s(\Delta)}\right|^2  ds+\frac{L_g}{2(\lambda_1-L_g)}\int^T_0|\rho_s|^2  ds.
\end{eqnarray*}
By Lemma \ref{COX}, we obtain
\begin{eqnarray*}
\EE\int^{T\wedge \tau^{\e}_R}_0|\Lambda_t|^{2}dt\leq \frac{L_g}{2(\lambda_1-L_g)}\EE\int^{T\wedge \tau^{\e}_R}_0|\rho_s|^2  ds+C_{R,T}\left(1+|x|^2+|y|^2\right)\Delta^{1/2}.\label{lambda}
\end{eqnarray*}
Now let's estimate term $V_t$:
\begin{eqnarray*}
 |V_t|\leq\!\!\!\!\!\!\!\!&&\frac{1}{\sqrt{\e \de}}  \int^t_0\left| e^{-\frac{\lambda_1 (t-s)}{\delta}}\sigma_2\left(X^{\e,\delta,u^{\e}}_s,Y^{\e,\delta,u^{\e}}_s\right)Q^{1/2}_2u^{\e}(s)\right|ds\\
=\!\!\!\!\!\!\!\!&&\frac{1}{\sqrt{\e\de}} \left(\int^t_0 e^{-\frac{2\lambda_1(t-s)}{\delta}}ds\right)^{\frac12}\cdot\left(\int_0^t\left|\sigma_2\left(X^{\e,\delta,u^{\e}}_s,Y^{\e,\delta,u^{\e}}_s\right)Q_2^{1/2}u^{\e}(s)\right|^2ds\right)^{\frac12}\\
\leq\!\!\!\!\!\!\!\!&& \frac{C\sqrt{\de}}{\sqrt{\e}}\left(\int^t_0\left(1+\left|X^{\e,\delta,u^{\e}}_s\right|^2\right)|u^{\e}(s)|^2 ds\right)^{\frac12}\\
\leq\!\!\!\!\!\!\!\!&& \frac{C\sqrt{\de}}{\sqrt{\e}}\left(1+\sup_{s\in [0, t]}\left|X^{\e,\delta,u^{\e}}_s\right|^2\right)^{\frac12}\left(\int^t_0|u^{\e}(s)|^2 ds\right)^{\frac12}.
\end{eqnarray*}
By the definition of $\tau^{\e}_R$, we have
\begin{eqnarray}
\EE\int^{T\wedge\tau^{\e}_R}_0|V_t|^{2}dt\leq \frac{C_{R,T}\sqrt{\de}}{\sqrt{\e}}\label{V}.
\end{eqnarray}

For term $M_t$, noting that $\frac{L^2_{\sigma_2}}{\lambda_1}+\frac{L_g}{\lambda_1-L_g}<1$,  there
exist  $\gamma_1, \gamma_2>1$ such that $\gamma_2\left(\frac{\gamma_{1}L^2_{\sigma_2}}{\lambda_1}+\frac{L_g}{\lambda_1-L_g}\right)<1$. Then, by Lemma \ref{COX},
\begin{eqnarray}
&&\EE\int^{T\wedge\tau^{\e}_R}_0|M_t|^{2}dt\nonumber\\
=\!\!\!\!\!\!\!\!&&\EE\int^{T\wedge\tau^{\e}_R}_0\left|\frac{1}{\sqrt{\de}}\int^t_0e^{\frac{(t-s)A}{\delta}}
\left[\sigma_2\left(X^{\e,\delta,u^{\e}}_s,Y^{\e,\delta,u^{\e}}_s\right)-\sigma_2\left(X^{\e,\delta,u^\e}_{s(\Delta)},\hat{Y}^{\e,\delta}_s\right)\right]Q_2^{1/2}dW_s\right|^2dt\nonumber\\
\leq\!\!\!\!\!\!\!\!&&\EE\int^{T}_0\left|\frac{1}{\sqrt{\de}}\int^{t\wedge \tau^{\e}_R}_0e^{\frac{(t-s)A}{\delta}}\left[\sigma_2\left(X^{\e,\delta,u^{\e}}_s,Y^{\e,\delta,u^{\e}}_s\right)-\sigma_2\left(X^{\e,\delta,u^\e}_{s(\Delta)},\hat{Y}^{\e,\delta}_s\right)\right]Q_2^{1/2}dW_s\right|^2dt\nonumber\\
\leq\!\!\!\!\!\!\!\!&&\frac{1}{\de}\EE\int^{T}_0\int^{t\wedge \tau^{\e}_R}_0e^{\frac{-2(t-s)\lambda_1}{\delta}}\left(C\left|X^{\e,\delta,u^{\e}}_s-X^{\e,\delta,u^\e}_{s(\Delta)}\right|^2+\gamma_{1}L^2_{\sigma_2}|\rho_s|^2\right)dsdt\nonumber\\
\leq\!\!\!\!\!\!\!\!&&\frac{1}{\de}\EE\int^{T\wedge\tau^{\e}_R}_0\left(\int^{T}_s e^{\frac{-2(t-s)\lambda_1}{\delta}}dt\right)\left(C\left|X^{\e,\delta,u^{\e}}_s-X^{\e,\delta,u^\e}_{s(\Delta)}\right|^2+\gamma_{1}L^2_{\sigma_2}\left|\rho_s\right|^2\right)ds\nonumber\\
\leq\!\!\!\!\!\!\!\!&&C_{\lambda_1}\EE\left[\int^{T\wedge\tau^{\e}_R}_0\left|X^{\e,\delta,u^{\e}}_s-X^{\e,\delta,u^\e}_{s(\Delta)}\right|^2ds\right]+\frac{\gamma_{1}L^2_{\sigma_2}}{2\lambda_1}\EE\int^{T\wedge\tau^{\e}_R}_0|\rho_s|^2ds\nonumber\\
\leq\!\!\!\!\!\!\!\!&&C_{\lambda_1,R,T}\left(1+|x|^2+|y|^2\right)\Delta^{1/2}+\frac{\gamma_{1}L^2_{\sigma_2}}{2\lambda_1}\EE\int^{T\wedge\tau^{\e}_R}_0|\rho_s|^2ds.\label{M}
\end{eqnarray}
Using the following inequality,
$$
\rho_t^2\le \gamma_2\left(\Lambda_t+M_t\right)^2+C_{\gamma_2}V_t^2\le 2\gamma_2\left(\Lambda_t^2+M_t^2\right)+C_{\gamma_2}V^2_t,
$$
and by \eqref{V} and \eqref{M}, we final obtain
\begin{align*}
 \EE\int^{T\wedge\tau^{\e}_R}_0|\rho_t|^{2}dt\le&    \gamma_2\left(\frac{\gamma_{1}L^2_{\sigma_2}}{\lambda_1}+\frac{L_g}{\lambda_1-L_g}\right) \EE\int^{T\wedge\tau^{\e}_R}_0|\rho_t|^{2}dt\\
 &+ C_{\lambda_1, R,T}\left(1+|x|^2+|y|^2\right)\Delta^{1/2}+\frac{C_{R,T}\sqrt{\de}}{\sqrt{ \e}},
\end{align*}
which implies \eqref{eq L3.1}. The proof is complete.
\end{proof}
\begin{remark}\label{Rem 4.6}
By comparing the equations of $Y^{\e,\delta, u^{\e}}_t$ and $\hat{Y}^{\e,\delta}_t$,  it is easy to see the additional controlled term including $u^\e$ in $Y^{\e,\delta, u^{\e}}_t$ disappears in $\hat{Y}^{\e,\delta}_t$. Lemma \ref{L3.1} implies  additional controlled term takes no effect as $\e\rightarrow 0$, which is the main reason why we assume \ref{A4} holds.
\end{remark}

Combining the following two propositions, we prove the convergence of controlled  sequence $X^{\e, \de, u^{\e}}$ to  averaged  process $\bar{X}^u$. This finally proves Condition (a) in Theorem \ref{thm BD}, so that the large derivation principle in the main result  Theorem \ref{main result 1} is obtained.
\begin{proposition}\label{convergence 1} For every fixed  $N\in\mathbb{N}$,  $\{u^\e\}_{\e>0} \in \mathcal{A}_N$,
	\vspace{2mm}
\begin{center}
    $X^{\e, \de, u^{\e}}-\hat{X}^{\e,\de}$ converges  to $0$ in
    distribution
\end{center}
\vspace{2mm}
in $C([0,T]; \HH)\cap L^2([0,T]; \VV)$ as $\e\rightarrow0$.
\end{proposition}
\begin{proof}
Define $Z^{\e, \de}_t:=X^{\e, \de, u^{\e}}_t-\hat{X}^{\e,\de}_t$. According to It\^{o}'s formula and Lemma \ref{Property B2}, we have
\begin{eqnarray*} \label{ItoFormu 01}
\left|Z^{\e, \de}_t\right|^{2}=\!\!\!\!\!\!\!\!&&-2\int^t_0 \left\|Z^{\e, \de}_s\right\|^{2}ds+2\int^t_0\left\langle B\left(X^{\e, \de, u^{\e}}_s\right)-B\left(\hat{X}^{\e,\de}_s\right), Z^{\e, \de}_s\right\rangle ds\\
&&+2\int^t_0 \left\langle f\left(X^{\e, \de, u^{\e}}_s, Y^{\e, \de, u^{\e}}_s\right)-f\left(X^{\e,\de,u^{\e}}_{s(\Delta)}, \hat{Y}^{\e,\de}_s\right), Z^{\e, \de}_s\right\rangle ds\\
&&+2\int^t_0 \left\langle \left[\sigma_1\left(X^{\e,\delta, u^{\e}}_s\right) -\sigma_1\left(\hat{X}^{\e,\de}_s\right)\right]Q_1^{1/2}u^{\e}(s), Z^{\e, \de}_s\right\rangle ds\\
&&+2\sqrt{\e}\int^t_0 \left\langle Z^{\e, \de}_s, \sigma_1\left(X^{\e,\delta, u^{\e}}_s\right)Q_1^{1/2}d W_s\right\rangle+\e \int^t_0 \left\|\sigma_1\left(X^{\e,\delta, u^{\e}}_s\right)Q_1^{1/2}\right\|_{\HS}^2 ds\\
\leq\!\!\!\!\!\!\!\!&&-2\int^t_0 \left\|Z^{\e, \de}_s\right \|^{2}ds+C\int^t_0 \left|Z^{\e, \de}_s\right|\cdot\left(\left\|X^{\e, \de, u^{\e}}_s\right\|+\left\|\hat X^{\e, \de}_s\right\|\right)\cdot\left\|Z^{\e, \de}_s\right\| ds\\
&&+C\int^t_0 \left|X^{\e, \de, u^{\e}}_s-X^{\e, \de, u^{\e}}_{s(\Delta)}\right|\cdot\left|Z^{\e, \de}_s\right|ds+C\int^t_0 \left|Y^{\e, \de, u^{\e}}_s-\hat{Y}^{\e,\de}_s\right|\cdot\left|Z^{\e, \de}_s\right|ds+C\int^t_0 \left|u^{\e}(s)\right|\cdot\left|Z^{\e, \de}_s\right|^2ds\\
&&+2\sqrt{\e}\int^t_0 \left\langle Z^{\e, \de}_s, \sigma_1\left(X^{\e,\delta, u^{\e}}_s\right)Q_{1}^{1/2}d W_s\right\rangle+\e \int^t_0\left \|\sigma_1\left(X^{\e,\delta, u^{\e}}_s\right)Q_1^{1/2}\right\|_{\HS}^2ds.
\end{eqnarray*}
By  Young's inequality,
\begin{eqnarray*}
\left|Z^{\e, \de}_t\right|^{2}+\int^t_0 \left\|Z^{\e, \de}_s\right \|^{2}ds\leq\!\!\!\!\!\!\!\!&&C\int^t_0\left |Z^{\e, \de}_s\right|^2\left(1+\left|u^{\e}(s)\right|^2+\left\|X^{\e, \de, u^{\e}}_s\right\|^2+\left\|\hat X^{\e, \de}_s\right\|^2\right) ds\\
&&+C\int^t_0\left|X^{\e, \de, u^{\e}}_s-X^{\e, \de, u^{\e}}_{s(\Delta)}\right|^2ds+C\int^t_0\left |Y^{\e, \de, u^{\e}}_s-\hat{Y}^{\e,\de}_s\right|^2ds\\
&&+2\sqrt{\e}\int^t_0 \left\langle Z^{\e, \de}_s, \sigma_1\left(X^{\e,\delta, u^{\e}}_s\right)Q_1^{1/2}d W_s\right\rangle+\e C\int^t_0 \left(1+\left|X^{\e,\delta, u^{\e}}_s\right|^2\right)ds.
\end{eqnarray*}

For any $\e, R>0$, we define a stopping time
\begin{equation}\label{stopping time}
\tilde{\tau}^{\e}_R:=\inf\left\{t>0, \left|X^{\e,\de,u^{\e}}_t\right|+\int^t_0 \left\|X^{\e,\de,u^{\e}}_s\right\|^2ds+\int^t_0\left\|\hat X^{\e, \de}_s\right\|^2 ds>R\right\}.
\end{equation}
By Gronwall's inequality,
\begin{eqnarray*}
&&\sup_{t\in [0, T\wedge \tilde \tau^{\e}_R]}\left|Z^{\e, \de}_t\right|^{2}+\int^{T\wedge \tilde \tau^{\e}_R}_0 \left\|Z^{\e, \de}_s \right\|^{2}ds\\
\leq\!\!\!\!\!\!\!\!&&\left[C\int^{T\wedge\tilde \tau^{\e}_R}_0 \!\!\left|Y^{\e, \de, u^{\e}}_s-\hat{Y}^{\e,\de}_s\right|^2ds\!\!+C\int^{T\wedge\tilde\tau^{\e}_R}_0\left|X^{\e, \de, u^{\e}}_s-X^{\e, \de, u^{\e}}_{s(\Delta)}\right|^2ds+\e C\int^{T\wedge\tilde \tau^{\e}_R}_0 \left(1+\left|X^{\e,\delta, u^{\e}}_s\right|^2\right)ds\right.\\
&&\left.+2\sqrt{\e}\sup_{t\in [0, T\wedge \tilde \tau^{\e}_R]}\left|\int^t_0 \left\langle Z^{\e, \de}_s, \sigma_1\left(X^{\e,\delta, u^{\e}}_s\right)Q_1^{1/2}d W_s\right\rangle\right|\right]e^{C_{R,T}}.
\end{eqnarray*}
Note that $\tilde \tau^{\e}_R\leq \tau^{\e}_R$, then by Lemmas \ref{COX} and \ref{L3.1}, and Burkholder-Davis-Gundy's inequality, it follows that
\begin{eqnarray}\label{F3.27}
\EE\left[\sup_{t\in [0, T\wedge\tilde \tau^{\e}_R]}\left|Z^{\e, \de}_t\right|^{2}\right]\!+\!\EE\int^{T\wedge\tilde \tau^{\e}_R}_0\left \|Z^{\e, \de}_s \right\|^{2}ds
\leq C_{R,T}(1+|x|^2+|y|^2)\left(\Delta^{1/2}\!+\!\frac{\sqrt{\de}}{\sqrt{ \e}}\!+\!\sqrt{\e}\right).
\end{eqnarray}

For any $r>0$, by the definition of stopping time $\tilde \tau^{\e}_R$ in \eqref{stopping time},  we have
\begin{eqnarray*}
&&\PP\left(\sup_{t\in [0,T]}\left|Z^{\e, \de}_t\right|^{2}+\int^T_0 \left\|Z^{\e, \de}_s\right \|^{2}ds\geq r\right)\\
\leq\!\!\!\!\!\!\!\!&&\PP\left(T>\tilde \tau^{\e}_R\right)+\PP\left(\sup_{t\in [0,T]}\left|Z^{\e, \de}_t\right|^{2}+\int^T_0 \left\|Z^{\e, \de}_s\right\|^{2}ds\geq r, T\leq\tilde \tau^{\e}_R\right)\\
\leq\!\!\!\!\!\!\!\!&&\PP\left(\sup_{t\in[0,T]}\left|X^{\e,\de,u^{\e}}_t\right|+\int^T_0 \left\|X^{\e,\de,u^{\e}}_s\right\|^2ds+\int^T_0 \left\|\hat X^{\e,\de}_s\right\|^2ds>R\right)\\
&&+\PP\left(\sup_{t\in[0, T\wedge\tilde \tau^{\e}_R]}\left|Z^{\e, \de}_t\right|^{2}+\int^{T\wedge\tilde \tau^{\e}_R}_0\left \|Z^{\e, \de}_s\right\|^{2}ds\geq r\right).
\end{eqnarray*}
By Lemma \ref{PE}, we can choose and fix $R$ large enough to make the first term on the right hand side of the above inequality small enough, and for fixed $R$ and \eqref{F3.27}, the  second term can also be small enough by choosing $\Delta=\delta^{1/2}$ and small $\varepsilon$.  Thus, we proved  $\sup_{t\leq T}\left|Z^{\e, \de}_t\right|^{2}+\int^T_0 \left\|Z^{\e, \de}_s\right \|^{2}ds$
converges to $0$ in probability. The proof is complete.
\end{proof}

\begin{proposition} \label{convergence 2} For any  fixed $N\in\mathbb{N}$,  let  $\{u^\e\}_{\e>0} \in \mathcal{A}_N$ such that that  $u^\e$ converges  to $u$ in
distribution, as $\e\rightarrow0$. It holds that
\vspace{2mm}
\begin{center}
    $\hat X^{\e, \de}-\bar{X}^u$ converges to $0$ in distribution
\end{center}
\vspace{2mm}
in $C([0,T]; \HH) $ as $\e\rightarrow0$, where $\bar{X}^u$  is the  solution to skeleton equation \eqref{eq sk}.
\end{proposition}
\begin{proof}

By the Skorokhod representation theorem,  we may assume that $u^{\e}\rightarrow u$ in $L^2([0,T]; \HH)$ almost surely in the weak topology.  The proof is divided into three steps.

\vspace{2mm}
\textbf{Step 1. (Splitting into three terms)}:
Let $\bar{Z}^{\e}_t:=\hat{X}^{\e,\delta}_t-\bar{X}^u_t$ and set $\bar{\Lambda}^{\e}_t:=\bar{Z}_t^{\e}-L_t^{\e}-N^{\e}_t$, where
$$
L_t^{\e}:=\int^t_0 e^{(t-s)A}\left[f\left(X^{\e,\delta,u^{\e}}_{s(\Delta)},Y^{\e,\delta}_s\right)-\bar{f}\left(\bar{X}^u_{s}\right)\right]ds,
$$
and
$$
N_t^{\e}:=\int^t_0 e^{(t-s)A}\sigma_1\left(\bar{X}^u_{s}\right)Q_1^{1/2}\left[u_s^{\e}-u_s\right]ds.
$$
Then it is easy to see $\bar{\Lambda}_t^{\e}$ satisfies the following equation
\begin{eqnarray*}
\frac{d\bar{\Lambda}^{\e}_t}{dt}=A\bar{\Lambda}^{\e}_t+\left[B\left(\hat{X}^{\e,\delta}_t\right)-B\left(\bar{X}^u_t\right)\right]+\left[\sigma_1\left(\hat{X}^{\e,\delta}_{t}\right)-\sigma_1\left(\bar{X}^u_{t}\right)\right]Q_1^{1/2}u^\e(t),\quad \bar{\Lambda}^\epsilon_0=0.
\end{eqnarray*}
By chain's  rule, we have
\begin{eqnarray*}
\left|\bar{\Lambda}_t^{\e}\right|^2=\!\!\!\!\!\!\!\!&&-2\int_{0}^{t}\left\|\bar{\Lambda}_s^{\e}\right\|^2ds
+2\int_{0}^{t}\left\langle B\left(\hat{X}^{\e,\delta}_s\right)-B\left(\bar{X}^u_s\right), \bar{\Lambda}_s^{\e}\right\rangle ds \nonumber\\
&&+2\int_{0}^{t}\left\langle \left[ \sigma_1\left(\hat{X}^{\e,\delta}_s\right) -\sigma_1\left(\bar{X}^u_s\right)\right]Q_1^{1/2}u^{\e}(s), \bar{\Lambda}^{\e}_s\right\rangle ds
\nonumber\\
\leq\!\!\!\!\!\!\!\!&&-2\int_{0}^{t}\left\|\bar{\Lambda}^{\e}_s\right\|^2ds
+2\int_{0}^{t}\left\|B\left(\hat{X}^{\e,\delta}_s\right)-B\left(\bar{X}^u_s\right)\right\|_{-1}\cdot \left\|\bar{\Lambda}^{\e}_s\right\| ds +C\int_{0}^{t}\left|\bar{Z}^{\e}_s\right|\cdot\left|u^{\e}(s)\right|\cdot\left|\bar{\Lambda}^{\e}_s\right| ds.
\end{eqnarray*}
Then by  Young's and  Poincar\'e's inequalities, we have
\begin{eqnarray*}
\left|\bar{\Lambda}_t^{\e}\right|^2\leq\!\!\!\!\!\!\!\!&&-2\int_{0}^{t}\left\|\bar{\Lambda}_s^{\e}\right\|^2ds
\!+\!C\int_{0}^{t}\left\|B\left(\hat{X}^{\e,\delta}_s\right)-B\left(\bar{X}^u_s\right)\right\|^2_{-1}ds\!+\!\int^t_0\left\|\bar{\Lambda}^{\e}_s\right\|^2 ds\!+\!C\int_{0}^{t}\left|\bar{Z}^{\e}_s|^2\cdot|u^{\e}(s)\right|^2ds\nonumber\\
\leq\!\!\!\!\!\!\!\!&&-\int_{0}^{t}\left\|\bar{\Lambda}_s^{\e}\right\|^2ds
+C\int_{0}^{t}\left|\bar{Z}^{\e}_s\right|^2\left (\left\|\hat{X}^{\e,\delta}_s\right\|^2+\left\|\bar{X}^u_s\right\|^2+\left|u^{\e}(s)\right|^2\right)ds.
\end{eqnarray*}
Then we obtain
\begin{eqnarray} \label{F3.35}
\sup_{t\in [0, T\wedge \tilde {\tau}^{\e}_{R}]} \left|\bar{\Lambda}^{\e}_t\right|^2\!\!+\int^{T\wedge \tilde {\tau}^{\e}_{R}}_0 \left\|\bar{\Lambda}^{\e}_s\right\|^2ds\leq\!\!\!\!\!\!\!\!&&C\int_{0}^{T\wedge \tilde {\tau}^{\e}_{R}}\left|\bar{Z}^{\e}_s\right|^2 \!\!\left(\left\|\hat{X}^{\e,\delta}_s\right\|^2+\left\|\bar{X}^u_s\right\|^2+\left|u^{\e}(s)\right|^2\right)ds.
\end{eqnarray}

\textbf{Step 2. (The estimate on $N^\e$)}:
For term $N^\e$,  we shall prove that it converges to $0$ in $C([0,T], \HH)$ almost surely, for which we firstly prove its tightness, and then  its convergence.

For any $\theta\in(0,1)$, by Lemma \ref{lem semigroup},  we have
\begin{align*}
\EE\left[\sup_{t\in [0,T]}\left\|N_t^{\e}\right\|^2_{\theta}\right]\le &\EE\left[\sup_{ t\in[0,T]}\int_0^t\left\|e^{(t-s)A}\sigma_1\left(\bar X^u_s\right)Q_1^{1/2}(u^\e_s-u_s)\right\|_{\theta}ds\right]^2\\
\le& \EE\left[\sup_{t\in [0,T]}\int_0^t (t-s)^{-\theta/2}\left|\sigma_1\left(\bar{X}^{u}_t\right)Q_1^{1/2}\left(u^\e_t-u_t\right)\right|ds\right]^2 \\
\le& C_T\EE\left[\left(1+\sup_{t\in [0,T]}\left|\bar{X}^{u}_t\right|^2\right)\cdot \int_0^T\left|u^\e_t-u_t\right|^2dt\right]\\
\le& C_{N, T}(1+|x|^2+|y|^2),
\end{align*}
where $C_{N, T}$ is independent of $\e$.

For any $0\le s\le t\le T$, by Lemma \ref{lem semigroup},  we have
\begin{align*}
&\EE\left[|N_t^{\e}-N_s^{\e}|^2\right]\\
=&\EE\left[\left|\int_0^t e^{(t-r)A}\sigma_1\left(\bar{X}^{u}_r\right)Q_1^{1/2}(u^\e_r-u_r)dr- \int_0^s e^{(s-r)A}\sigma_1\left(\bar{X}^{u}_r\right)Q_1^{1/2}\left(u^\e_r-u_r\right)dr\right|^2\right]\\
\le&2 \EE\left[\left|\int_s^t e^{(t-r)A }  \sigma_1\left(\bar{X}^{u}_r\right)Q_1^{1/2}\left(u^\e_r-u_r\right)dr \right|^2\right]\\
&+2 \EE\left[\left|\int_0^s\left[ e^{(t-r)A } -e^{(s-r)A}\right] \sigma_1\left(\bar{X}^{u}_r\right)Q_1^{1/2}\left(u^\e_r-u_r\right)dr \right|^2\right]\\
\leq& C\EE\left[\left(1+\sup_{ r\in [0, T]}\left|\bar{X}^u_r\right|^2\right) \cdot \int_0^T \left|u^\e_r-u_r\right|^2dr\right]|t-s|\\
& +C\left[\int_0^s\frac{(t-s)^{1/2}}{(s-r)^{1/2}}dr\right] \EE\left[\left(1+\sup_{r\in [0,T]}\left|\bar{X}^u_r\right|^2\right) \cdot \int_0^T \left|u^\e_r-u_r\right|^2dr\right]\\
\leq & C_T(1+|x|^2+|y|^2)|t-s|^{\frac{1}{2}}.
\end{align*}
Applying an Arzela-Ascoli's argument, we can show that $\{N^{\e}\}_{\e\in(0,1]}$ is tight in $C([0,T];\HH)$.  Thus, there exist a subsequence $\{N^{\e_n}\}_{n\ge1}$  being the Cauchy sequence, whose limit is denoted by $N^0$.
By chain rule, we know that
\begin{align*}
&\left|N^{\e_n}_t\right|^2+2\int_0^t\left\|N^{\e_n}_s\right\|^2ds\\
=&2\left|\int_0^t\left\langle N^{\e_n}_s, \sigma_1\left(\bar{X}^{u}_s\right)Q_1^{1/2}(u^{\e_n}_s-u_s)\right\rangle ds\right|\\
\le &2\left|\int_0^t\left\langle N^{\e_n}_s-N^0_s, \sigma_1\left(\bar{X}^{u}_s\right)Q_1^{1/2}\left(u^{\e_n}_s-u_s\right)\right\rangle ds\right|+2\left|\int_0^t\left\langle N^0_s, \sigma_1\left(\bar{X}^{u}_s\right)Q_1^{1/2}(u^{\e_n}_s-u_s)\right\rangle ds\right|\\
\le & 2\sup_{s\in [0,t]}\left|N^{\e_n}_s-N^0_s\right| \cdot\int_0^t\left(1+\left|\bar{X}^{u}_s\right|\right) \cdot\left|u^{\e_n}_s-u_s\right|ds +2\left|\int_0^t\left\langle Q_1^{1/2}\sigma_1^*\left(\bar{X}^{u}_s\right) N^0_s, u^{\e_n}_s-u_s\right\rangle ds\right|\\
&  \longrightarrow 0,\ \ \text{a.s.},
\end{align*}
where $\sigma_1^*$ is the transpose of $\sigma_1$ and we have used the facts of $N^{\e_n}\rightarrow N^0$ in $C([0,T], \HH)$, $u^{\e_n}\rightarrow u$ in $\mathbb S_N$ and $Q_1^{1/2}\sigma_1^*(\bar{X}^{u}) N^0$ belongs to   $L^2([0,T];\mathbb H)$. By the uniqueness of the limit, we know that  $N^\e\rightarrow 0$ in $C([0,T]; \HH)$ almost surely.

\textbf{Step 3. (The estimate on $L_t^{\e}$):}
For term $L_t^{\e}$,
\begin{eqnarray*}
L_t^{\e}=\!\!\!\!\!\!\!\!&&\int^t_0e^{(t-s)A}\left[f\left(X_{s(\Delta)}^{\e,\de,u^{\e}},\hat{Y}_{s}^{\e,\de}\right)-\bar{f}\left(X_{s}^{\e,\de,u^{\e}}\right)\right]ds\\
&&+\int^t_0e^{(t-s)A}\left[\bar{f}\left(X_{s}^{\e,\de,u^{\e}}\right)-\bar{f}\left(\hat{X}_{s}^{\e,\de}\right)\right]ds+\int^t_0e^{(t-s)A}\left[\bar{f}\left(\hat{X}_{s}^{\e,\de}\right)-\bar{f}\left(\bar{X}^u_{s}\right)\right]ds\\
=:\!\!\!\!\!\!\!\!&&I^{\e}_1(t)+I^{\e}_2(t)+I^{\e}_3(t).
\end{eqnarray*}
By Lipschitz property of $\bar{f}$, we have
\begin{eqnarray*}
\sup_{t\in[0,T\wedge\tilde {\tau}^{\e}_{R}]}\left|I^{\e}_2(t)\right|^2
\leq C\int^{T\wedge\tilde {\tau}^{\e}_{R}}_0\left|\bar{f}\left(X_{s}^{\e,\de,u^{\e}}\right)-\bar{f}\left(\hat{X}_{s}^{\e,\de}\right)\right|^2ds
\leq C\int^{T\wedge\tilde {\tau}^{\e}_{R}}_0\left|  X_{s}^{\e,\de,u^{\e}}-\hat{X}_{s}^{\e,\de}  \right|^2ds,
\end{eqnarray*}
and
\begin{eqnarray*}
\sup_{t\in [0,T\wedge\tilde {\tau}^{\e}_{R}]}\left|I^{\e}_3(t)\right|^2
\le C\int^{T\wedge\tilde {\tau}^{\e}_{R}}_0\left|\bar{f}\left(\hat{X}_{s}^{\e,\de}\right)-\bar{f}\left(\bar{X}^u_{s}\right)\right|^2ds\le
 C\int^{T\wedge\tilde{\tau}^{\e}_{R}}_0\left|\bar{Z}^{\e}_s\right|^2ds.
\end{eqnarray*}
Then it follows that
\begin{eqnarray}\label{F3.37}
\sup_{t\in [0,T\wedge\tilde {\tau}^{\e}_{R}]}\left|L_t^{\e}\right|^2\leq C\sup_{t\in [0,T\wedge\tilde {\tau}^{\e}_{R}]}\left|I^{\e}_1(t)\right|^2+C\int^{T\wedge\tilde {\tau}^{\e}_{R}}_0\left|X_{s}^{\e,\de,u^{\e}}-\hat{X}_{s}^{\e,\de}\right|^2ds
 +C\int^{T\wedge\tilde {\tau}^{\e}_{R}}_0\left|\bar{Z}^{\e}_s\right|^2ds.\nonumber\\
\end{eqnarray}
By \eref{F3.35} and \eref{F3.37}, we obtain
\begin{eqnarray*}
\sup_{t\in [0,T\wedge\tilde{\tau}^{\e}_{R}]}\left|\bar{Z}^{\e}_t\right|^2\leq\!\!\!\!\!\!\!\!&&C\int_{0}^{T\wedge \tilde{\tau}^{\e}_{R}} \left|\bar{Z}^{\e}_s\right|^2\left(1+\left\|\hat{X}^{\e,\delta}_s\right\|^2+\left\|\bar{X}^u_s\right\|^2+\left|u^{\e}(s)\right|^2\right)ds\nonumber\\
&&+ C\sup_{t\in [0,T\wedge\tilde{\tau}^{\e}_{R}]}\left|I_1^{\e}(t)\right|^2+C\sup_{t\in [0,T]}\left|N^{\e}_t\right|^2+C\int^{T\wedge\tilde {\tau}^{\e}_{R}}_0\left|X_{s}^{\e,\de,u^{\e}}-\hat{X}_{s}^{\e,\de}\right|^2ds.
\end{eqnarray*}
By \eref{eq skeleton estimate} and the definition $\tilde\tau^{\e}_R$,   Gronwall's inequality implies that
\begin{eqnarray}
\sup_{t\in [0,T\wedge\tilde{\tau}^{\e}_{R}]}\left|\bar{Z}^{\e}_t\right|^2\leq\!\!\!\!\!\!\!\!&&\!\!\left[\sup_{t\in [0,T\wedge\tilde{\tau}^{\e}_{R}]}|I_1^{\e}(t)|^2 \!\!+\sup_{t\in [0,T]}\left|N^{\e}_t\right|^2\!\!+\!\!\int^{T\wedge\tilde{\tau}^{\e}_{R}}_0\left|X_{s}^{\e,\de,u^{\e}}-\hat{X}_{s}^{\e,\de}\right|^2ds\right]e^{C_{R,N,T}}.\label{bar Z}
\end{eqnarray}

 Next, we estimate $I_1^{\e}(t)$. Let $n_{t}:=\left[\frac{t}{\Delta}\right]$. Denote
  \begin{align*}
I_1^{\e}(t)=J_1^{\e}(t)+J_2^{\e}(t)+J_3^{\e}(t),
\end{align*}
where
\begin{align*}
J_1^{\e}(t):=\sum_{k=0}^{n_{t}-1}
\int_{k\Delta}^{(k+1)\Delta}e^{(t-s)A}\left[f\left(X_{k\Delta}^{\e,\de,u^{\e}},\hat{Y}_{s}^{\e,\de}\right)-\bar{f}\left(X_{k\Delta}^{\e,\de,u^{\e}}\right)\right]ds,
\end{align*}
\begin{align*}
J_2^{\e}(t):=\sum_{k=0}^{n_{t}-1}
\int_{k\Delta}^{(k+1)\Delta}e^{(t-s)A}\left[\bar{f}\left(X_{k\Delta}^{\e,\de,u^{\e}}\right)-\bar{f}\left(X_{s}^{\e,\de,u^{\e}}\right)\right]ds,
\end{align*}
\begin{align*}
J_3^{\e}(t):=
\int_{n_{t}\Delta}^{t}e^{(t-s)A}\left[f\left(X_{n_{t}\Delta}^{\e,\de,u^{\e}},\hat{Y}_{s}^{\e,\de}\right)-\bar{f}\left({X}_{s}^{\e,\delta, u^\e}\right)\right]ds.
\end{align*}

For $J_2^{\e}(t)$,  noticing that $\tilde\tau^{\e}_R\leq \tau^{\e}_R$, by Lemma \ref{COX}, we have
\begin{eqnarray}   \label{J42}
\EE\left(\sup_{t\in [0, T\wedge \tilde{\tau}^{\e}_{R}]} |J^{\e}_{2}(t)|^2\right)
\leq\!\!\!\!\!\!\!\!&&C\EE\int_{0}^{T\wedge \tilde{\tau}^{\e}_{R}}\left|X_{s(\Delta)}^{\e,\de,u^{\e}}-X_{s}^{\e,\de,u^{\e}}\right|^2ds\nonumber\\
\leq\!\!\!\!\!\!\!\!&&C_{R,T}\left(1+|x|^2+|y|^2\right)\Delta^{1/2}.
\end{eqnarray}

For the term $J_3^{\e}(t)$,  by \eref{Control X} and \eref{hat Y}, we have
 \begin{eqnarray}  \label{J43}
\EE\left(\sup_{ t\in [0, T\wedge \tilde{\tau}^{\e}_{R}]} | J^{\e}_3(t)|^2\right)\leq\!\!\!\!\!\!\!\!&&C\Delta\int_{n_{t}\Delta}^{t}\EE\left(1+\left|X_{n_{t}\Delta}^{\e,\de,u^{\e}}\right|^2+\left|\hat{Y}_{s}^{\e,\de}\right|^2+\left|X_{s}^{\e,\de,u^{\e}}\right|^2\right)ds\nonumber\\
\leq\!\!\!\!\!\!\!\!&& C_T\left(1+ |x |^{2}+ |y|^{2}\right)\Delta^2.
\end{eqnarray}

For the term $J_1^{\e}(t)$,
by the construction of $\hat{Y}_{t}^{\e,\de}$,
we obtain that, for any $k\in \mathbb{N}_{\ast}$ and $s\in[0,\Delta)$,
\begin{eqnarray}
\hat{Y}_{s+k\Delta}^{\e,\de}
=\!\!\!\!\!\!\!\!&&\hat Y_{k\Delta}^{\e,\de}+\frac{1}{\de}\int_{k\Delta}^{k\Delta+s}A\hat{Y}_{r}^{\e,\de}dr
+\frac{1}{\de}\int_{k\Delta}^{k\Delta+s}g\left(X_{k\Delta}^{\e,\de,u^{\e}},\hat{Y}_{r}^{\e,\de}\right)dr\nonumber\\
 \!\!\!\!\!\!\!\! &&
+\frac{1}{\sqrt{\de}}\int_{k\Delta}^{k\Delta+s}\sigma_2\left(X_{k\Delta}^{\e,\de,u^{\e}},\hat{Y}_{r}^{\e,\de}\right)Q_2^{1/2}dW_r   \nonumber\\
=\!\!\!\!\!\!\!\!&&\hat Y_{k\Delta}^{\e,\de}+\!\!\frac{1}{\de}\int_{0}^{s}A\hat{Y}_{r+k\Delta}^{\e,\de}dr
+\!\!\frac{1}{\de}\int_{0}^{s}g\left(X_{k\Delta}^{\e,\de,u^{\e}},\hat{Y}_{r+k\Delta}^{\e,\de}\right)dr\nonumber\\
\!\!\!\!\!\!\!\! &&
+\!\!\frac{1}{\sqrt{\de}}\int_{0}^{s}\sigma_2\left(X_{k\Delta}^{\e,\de,u^{\e}},\hat{Y}_{r+k\Delta}^{\e,\de}\right)Q_2^{1/2}d\bar{W} _r, \label{E3.26}
\end{eqnarray}
where $\bar{W}_t:=W_{t+k\Delta}-W_{k\Delta}$ is the shift version of $W_t$.
Recall that $\tilde{W}_t$ is a standard  cylindrical  Wiener process independent of $\left(X_{k\Delta}^{\e,\de,u^{\e}},\hat Y_{k\Delta}^{\e,\de}\right)$.
Denote by $\hat W_t=\de^{1/ 2}\tilde{W}_{\frac{t}{\de}}$.
We construct a process $Y^{X_{k\Delta}^{\e,\de,u^{\e}},\hat Y_{k\Delta}^{\e,\de}}_t$ by means of $Y^{x,y}_t \mid_{(x,y)=\left(X_{k\Delta}^{\e,\de,u^{\e}},\hat Y_{k\Delta}^{\e,\de}\right)}$, where $Y^{x,y}$ is the solution to Eq. \eqref{FEQ}. Specifically, that is
\begin{eqnarray}
Y_{\frac{s}{\de}}^{X_{k\Delta}^{\e,\de,u^{\e}},\hat Y_{k\Delta}^{\e,\de}}=\!\!\!\!\!\!\!\!&&\hat Y_{k\Delta}^{\e,\de}
+\int_{0}^{\frac{s}{\de}}AY_{r}^{X_{k\Delta}^{\e,\de,u^{\e}},\hat Y_{k\Delta}^{\e,\de}}dr
+\int_{0}^{\frac{s}{\de}}g\left(X_{k\Delta}^{\e,\de,u^{\e}}, Y_{r}^{X_{k\Delta}^{\e,\de,u^{\e}},\hat Y_{k\Delta}^{\e,\de}}\right)dr\nonumber\\
\!\!\!\!\!\!\!\!&&+\int_{0}^{\frac{s}{\de}}\sigma_2\left(X_{k\Delta}^{\e,\de,u^{\e}}, Y_{r}^{X_{k\Delta}^{\e,\de,u^{\e}},\hat Y_{k\Delta}^{\e,\de}}\right)Q_2^{1/2}d\tilde{W} _r \nonumber\\
=\!\!\!\!\!\!\!\!&&\hat Y_{k\Delta}^{\e,\de}
+\frac{1}{\de}\int_{0}^{s}AY_{\frac{r}{\de}}^{X_{k\Delta}^{\e,\de,u^{\e}},\hat Y_{k\Delta}^{\e,\de}}dr
+\frac{1}{\de}\int_{0}^{s}g\left(X_{k\Delta}^{\e,\de,u^{\e}},Y_{\frac{r}{\de}}^{X_{k\Delta}^{\e,\de,u^{\e}},\hat Y_{k\Delta}^{\e,\de}}\right)dr\nonumber\\
&&+\frac{1}{\sqrt{\de}}\int_{0}^{s}\sigma_2\left(X_{k\Delta}^{\e,\de,u^{\e}}, Y_{\frac{r}{\delta}}^{X_{k\Delta}^{\e,\de,u^{\e}},\hat Y_{k\Delta}^{\e,\de}}\right)Q_2^{1/2}d\hat W_r.  \label{E3.27}
\end{eqnarray}
The uniqueness of the solutions to Eq. \eref{E3.26} and Eq. \eref{E3.27} implies
that the distribution of $\left(X_{k\Delta}^{\e,\de,u^{\e}},\hat{Y}^{\e,\de}_{s+k\Delta}\right)$
coincides with the distribution of
$\left(X_{k\Delta}^{\e,\de,u^{\e}},
Y_{\frac{s}{\de}}^{X_{k\Delta}^{\e,\de,u^{\e}},\hat Y_{k\Delta}^{\e,\de}}\right)$.

Next,  we try to control  $\left|J_1^{\e}(t)\right|$:
\begin{eqnarray*}
&&\EE\left[\sup_{t\in [0,T]} \left| J_1^{\e}(t) \right|^{2} \right] \nonumber\\
=\!\!\!\!\!\!\!\!&&
\EE\sup_{t\in [0,T]}\Big |\sum_{k=0}^{n_{t}-1}e^{(t-(k+1)\Delta)A}
\int_{k\Delta}^{(k+1)\Delta}e^{((k+1)\Delta-s)A}
\left[f\left(X_{k\Delta}^{\e,\de,u^{\e}},\hat{Y}_{s}^{\e,\de}\right)
- \bar{f}\left(X_{k\Delta}^{\e,\de,u^{\e}}\right)\right]ds\Big |^{2}
\nonumber\\
\leq\!\!\!\!\!\!\!\!&& \EE\sup_{t\in [0,T]}
\left\{n_{t}\sum_{k=0}^{n_{t}-1}\Big |\int_{k\Delta}^{(k+1)\Delta}
e^{((k+1)\Delta-s)A}\left[f\left(X_{k\Delta}^{\e,\de,u^{\e}},\hat{Y}_{s}^{\e,\de}\right)-\bar{f}\left(X_{k\Delta}^{\e,\de,u^{\e}}\right)\right]ds\Big |^{2}\right\}
\nonumber\\
\leq\!\!\!\!\!\!\!\!&&
\left[\frac{T}{\Delta}\right]
\sum_{k=0}^{\left[\frac{T}{\Delta}\right]-1}
\mathbb{E} \Big |\int_{k\Delta}^{(k+1)\Delta}
e^{((k+1)\Delta-s)A}\left[f\left(X_{k\Delta}^{\e,\de,u^{\e}},\hat{Y}_{s}^{\e,\de}\right)-\bar{f}\left(X_{k\Delta}^{\e,\de,u^{\e}}\right)\right]ds\Big|^{2}
\nonumber\\
\leq\!\!\!\!\!\!\!\!&&
\frac{C_{T}}{\Delta^{2}}\max_{0\leq k\leq\left[\frac{T}{\Delta}\right]-1}\mathbb{E}
\Big | \int_{k\Delta}^{(k+1)\Delta}
e^{((k+1)\Delta-s)A}\left[f\left(X_{k\Delta}^{\e,\de,u^{\e}},\hat{Y}_{s}^{\e,\de}\right)-\bar{f}\left(X_{k\Delta}^{\e,\de,u^{\e}}\right)\right]ds \Big |^{2}.  \nonumber\\
\end{eqnarray*}
Then by changing variable, we get
\begin{eqnarray*}
\EE\left[\sup_{t\in [0,T]} \left| J_1^{\e}(t) \right|^{2} \right]\leq\!\!\!\!\!\!\!\!&&
C_{T}\frac{\de^{2}}{\Delta^{2}}\max_{0\leq k\leq\left[\frac{T}{\Delta}\right]-1}
\mathbb{E}
\Big| \int_{0}^{\frac{\Delta}{\de}}
e^{(\Delta-s\de)A}
\left[f\left(X_{k\Delta}^{\e,\de,u^{\e}},\hat{Y}_{s\de+k\Delta}^{\e,\de}\right)-\bar{f}\left(X_{k\Delta}^{\e,\de,u^{\e}}\right)\right]ds\Big|^{2}  \nonumber\\
=\!\!\!\!\!\!\!\!&&2C_{T}\frac{\de^{2}}{\Delta^{2}}\max_{0\leq k\leq\left[\frac{T}{\Delta}\right]-1}\int_{0}^{\frac{\Delta}{\de}}
\int_{r}^{\frac{\Delta}{\de}}\Psi_{k}(s,r)dsdr,  \nonumber
\end{eqnarray*}
where
\begin{eqnarray*}
&&\Psi_{k}(s,r)\\
=\!\!\!\!\!\!\!\!&&\mathbb{E}\left\langle e^{(\Delta-s\de)A}
\big[f\left(X_{k\Delta}^{\e,\de,u^{\e}},\hat{Y}_{s\de+k\Delta}^{\e,\de}\right)-\bar{f}\left(X_{k\Delta}^{\e,\de,u^{\e}}\right)\big], e^{(\Delta-r\de)A}
\big[f\left(X_{k\Delta}^{\e,\de,u^{\e}},\hat{Y}_{r\de+k\Delta}^{\e,\de}\right)-\bar{f}\left(X_{k\Delta}^{\e,\de,u^{\e}}\right)\big]\right\rangle  \nonumber\\
=\!\!\!\!\!\!\!\!&&\mathbb{E}\left\langle e^{\left(\Delta-s\de\right)A}
\left[f\left(X_{k\Delta}^{\e,\de,u^{\e}},Y_{s}^{X_{k\Delta}^{\e,\de,u^{\e}},\hat Y_{k\Delta}^{\e,\de}}\right)-\bar{f}\left(X_{k\Delta}^{\e,\de,u^{\e}}\right)\right]\right.,\nonumber\\
&&\hspace{2cm}\left. e^{(\Delta-r\de)A}\left[f\left(X_{k\Delta}^{\e,\de,u^{\e}}, Y_{r}^{X_{k\Delta}^{\e,\de,u^{\e}},\hat Y_{k\Delta}^{\e,\de}}\right)-\bar{f}\left(X_{k\Delta}^{\e,\de,u^{\e}}\right)\right]\right\rangle.  \nonumber
\end{eqnarray*}
Now, let's estimate $\Psi_{k}(s,r)$. Define $\tilde {\mathcal{F}}_s:=\sigma\{ Y_{u}^{x,y},u\leq s\}.$ Then for $s>r$, by the Markov property and Proposition \ref{ergodicity},
\begin{eqnarray*}
&&\Psi_{k}(s,r)\\
=\!\!\!\!\!\!\!\!&&\mathbb{E}\left\{\EE\left\langle e^{(\Delta-s\de)A}
\big[f\left(x,Y_{s}^{x,y}\right)-\bar{f}(x)\big], e^{(\Delta-r\de)A}\left[f\left(x, Y_{r}^{x,y}\right)-\bar{f}(x)\right]\right\rangle \mid_{(x,y)=\left(X_{k\Delta}^{\e,\de,u^{\e}},\hat{Y}^{\e,\de}_{s+k\Delta}\right)}\right\}  \nonumber\\
=\!\!\!\!\!\!\!\!&&\mathbb{E}\left\{\EE\Big[\left\langle e^{(\Delta-s\de)A}
\EE \big[f\left(x,Y_{s}^{x,y}\right)-\bar{f}(x)\mid \tilde {\mathcal{F}}_{r}\big], e^{(\Delta-r\de)A}\big[f\left(x, Y_{r}^{x,y}\right)-\bar{f}(x)\big]\right\rangle\Big]\mid_{(x,y)=\left(X_{k\Delta}^{\de},\hat{Y}^{\e,\de}_{s+k\Delta}\right)}\right\}\nonumber\\
\leq\!\!\!\!\!\!\!\!&&C\mathbb{E}\left\{\EE\left[\left|\EE f\left(x,Y_{s-r}^{x,z}\right)-\bar{f}(x)\right|1_{\{z=Y^{x,y}_r\}}\left(1+|x|+|Y_{r}^{x,y}|\right)\right]\mid_{(x,y)=\left(X_{k\Delta}^{\e,\de,u^{\e}},\hat Y_{k\Delta}^{\e,\de}\right)}\right\}\nonumber\\
\leq\!\!\!\!\!\!\!\!&& C\mathbb{E}\left[\EE\left(1+|x|^2+\left|Y^{x,y}_{r}\right|^2\right)e^{-(s-r)\eta}\mid_{(x,y)=\left(X_{k\Delta}^{\e,\de,u^{\e}},\hat Y_{k\Delta}^{\e,\de}\right)}\right]\nonumber\\
\leq\!\!\!\!\!\!\!\!&&C\mathbb{E}\left(1+\left|X_{k\Delta}^{\e,\de,u^{\e}}\right|^2+\left|\hat Y^{\e,\de}_{k\Delta}\right|^2\right)e^{-(s-r)\eta}\nonumber\\
\leq\!\!\!\!\!\!\!\!&&C_T\left(1+ |x|^2+ |y|^2\right)e^{-(s-r)\eta},
\end{eqnarray*}
where the last two inequalities are deduced by \eref{Control X} and \eref{hat Y}. Then we get
\begin{eqnarray}  \label{J412}
\mathbb{E}\left[\sup_{t\in [0,T]} \left| J_1^{\e}(t)\right|^{2}\right]
\leq\!\!\!\!\!\!\!\!&&
C_{T}\frac{\de^{2}}{\Delta^{2}}\left(1 + |x|^2 + |y|^2\right)
\int_{0}^{\frac{\Delta}{\de}}\int_{r}^{\frac{\Delta}{\de}}e^{-\frac{1}{2}(s-r)\eta}dsdr    \nonumber\\
\leq\!\!\!\!\!\!\!\!&&C_{T}\frac{\de}{\Delta}\left(1+ |x|^2+ |y|^2\right).
\end{eqnarray}
Thus, combining  \eref{J42}, \eref{J43} and \eref{J412}, we get
\begin{eqnarray}
\mathbb{E}\left[\sup_{t\in [0, T\wedge \tilde{\tau}^{\e}_{R}]} | I_1^{\e}(t) |^{2}\right]
\leq\!\!\!\!\!\!\!\!&&C_{R,T}\left(1+|x|^2+|y|^2\right)\Big(\Delta^{1/2}+\frac{\de}{\Delta}\Big).  \label{I1}
\end{eqnarray}
According to the estimates (\ref{bar Z}) and (\ref{I1}), we obtain
\begin{align*}
 &\mathbb{E}\Big[\sup_{t\in [0, T\wedge \tilde{\tau}^{\e}_{R}]}
\left| \hat{X}_{t}^{\e,\de}-\bar{X}^u_{t}\right|^{2}\Big]\\
\leq &
C_{R,N,T}\left(1+|x|^2+|y|^2\right)\left(\Delta^{1/2}+\frac{\de}{\Delta}\right)+C\mathbb{E}\int^{T\wedge\tilde {\tau}^{\e}_{R}}_0\left|X_{s}^{\e,\de,u^{\e}}-\hat{X}_{s}^{\e,\de}\right|^2ds.
\end{align*}
By \eref{F3.27} and choosing $\Delta=\delta^{1/2}$, we have
\begin{eqnarray}
\lim_{\e\rightarrow 0}\mathbb{E}\Big[\sup_{ t\in [0, T\wedge \tilde{\tau}^{\e}_{R}]}
\left| \hat{X}_{t}^{\e,\de}-\bar{X}^u_{t}\right|^2\Big]=0. \label{F3.45}
\end{eqnarray}
 For any $r>0$, by the definition of stopping time $\tilde{\tau}^{\e}_R$ and \eref{F3.45}
\begin{eqnarray*}
\PP\left(\sup_{t\in [0,T]}\left| \hat{X}_{t}^{\e,\de}-\bar{X}^u_{t}\right|\geq r\right)\leq\!\!\!\!\!\!\!\!&&\PP\left(T>\tilde{\tau}^{\e}_R\right)+\PP\left(\sup_{t\in [0,T]}\left| \hat{X}_{t}^{\e,\de}-\bar{X}^u_{t}\right|\geq r, T\leq \tilde{\tau}^{\e}_R\right)\\
\leq\!\!\!\!\!\!\!\!&&\PP\left(\sup_{t\in[0,T]}\left|X^{\e,\de,u^{\e}}_t\right|+\int^T_0 \left\|X^{\e,\de,u^{\e}}_s\right\|^2ds+\int^T_0\left\|\hat X^{\e,\de}_s\right\|^2ds>R\right)\\
&&+\PP\left(\sup_{t\in [0, T\wedge\tilde{\tau}^{\e}_R]}\left| \hat{X}_{t}^{\e,\de}-\bar{X}^u_{t}\right|\geq r\right).\end{eqnarray*}
By Lemmas \ref{PE} and \ref{AE}, we can choose an fixed $R$ large enough to make the first term on the right hand side of the above inequality small enough, and for fixed $R$ and \eqref{F3.45}, the  second term can also be small enough by choosing small $\e$.  Thus, we proved $\sup_{t\leq T}\left| \hat{X}_{t}^{\e,\de}-\bar{X}^u_{t}\right|\rightarrow 0$ in probability.
The proof is complete.
\end{proof}

\section{Appendix}

\subsection{A weak convergence criteria for LDP}

In this part, we will recall the general criteria for a LDP given in \cite{Budhiraja-Dupuis}.
Let $(\Omega,\mathcal{F},\mathbb{P})$ be a probability space with an increasing family $\{\FF_t\}_{0\le t\le T}$ of the sub-$\sigma$-fields of $\FF$ satisfying the usual conditions.
Let $\mathcal{E}$ be a Polish space with the Borel $\sigma$-field $\mathcal{B}(\mathcal{E})$.
\begin{definition}\label{Dfn-Rate function}
	\emph{\textbf{(Rate function)}} A function $I: \mathcal{E}\rightarrow[0,\infty]$ is called a rate function on
	$\mathcal{E}$,
	if for each $M<\infty$, the level set $\{x\in\mathcal{E}:I(x)\leq M\}$ is a compact subset of $\mathcal{E}$.
\end{definition}
\begin{definition}
	\emph{\textbf{(LDP)}} Let $I$ be a rate function on $\mathcal{E}$.  A family
	$\{X^\e \}$ of $\mathcal E$-valued random elements is  said to satisfy the LDP on $\mathcal{E}$
	with rate function $I$, if the following two conditions
	hold.
	\begin{itemize}
		\item[$(a)$](Upper bound) For each closed subset $F$ of $\mathcal{E}$,
		$$
		\limsup_{\e \rightarrow 0}\e \log\mathbb{P}(X^\e \in F)\leq- \inf_{x\in F}I(x).
		$$
		\item[$(b)$](Lower bound) For each open subset $G$ of $\mathcal{E}$,
		$$
		\liminf_{\e \rightarrow 0}\e \log\mathbb{P}(X^\e \in G)\geq- \inf_{x\in G}I(x).
		$$
	\end{itemize}
\end{definition}

\vskip0.3cm

Let $\mathcal{A}$ denote  the class of  $\{\FF_t\}$-predictable processes $u$ belonging to $\mathbb S$ a.s..
Let $\mathbb S_N=\{u\in L^2([0, T],\HH); \int_0^T|u(s)|^2ds\le N\}$. The set $\mathbb S_N$ endowed with the weak topology is a Polish space.
Define $\mathcal{A}_N=\{\phi\in \mathcal{A};u(\omega)\in \mathbb S_N, \mathbb{P}\text{-a.s.}\}$.

Recall the following result from Budhiraja and Dupuis \cite{Budhiraja-Dupuis}.

\begin{theorem}\label{thm BD} (\cite{Budhiraja-Dupuis}) Let  $\{\Gamma^\e\}_{\e>0}$ be a family of  measurable mappings from $C([0,T], \HH)$ into $\mathcal{E}$.
	Suppose that there  exists a measurable map $\Gamma^0:C([0,T], \HH)\rightarrow \mathcal{E}$ such that
	\begin{itemize}
		\item[(a)] for every $N<+\infty$ and any family $\{u^\e;\e>0\}\subset \mathcal{A}_N$ satisfying that $u^\e$ converges in distribution as $\mathbb S_N$-valued random elements to $u$ as $\e\rightarrow 0$,
		$\Gamma^\e\left(W(\cdot)+\frac{1}{\sqrt\e}\int_0^{\cdot}u^\e(s)ds\right)$ converges in distribution to $\Gamma^0(\int_0^{\cdot}u(s)ds)$ as $\e\rightarrow 0$;
		\item[(b)] for every $N<+\infty$, the set
		$
		\left\{\Gamma^0\left(\int_0^{\cdot}u(s)ds\right); u\in \mathbb S_N\right\}
		$
		is a compact subset of $\mathcal{E}$.
	\end{itemize}
	Then the family $\{\Gamma^\e(W)\}_{\e>0}$ satisfies a LDP in $\mathcal E$ with the rate function $I$ given by
	\begin{equation}\label{rate function}
	I(g):=\inf_{\left\{u\in \mathbb S; g=\Gamma^0\left(\int_0^{\cdot}u(s)ds\right)\right\}}\left\{\frac12\int_0^T|u(s)|^2ds\right\},\ g\in\mathcal{E},
	\end{equation}
	with the convention $\inf \emptyset=\infty$.
	
\end{theorem}

\subsection{Some estimates about the Burgers equation}  We recall some  properties of the semigroup  $\{e^{tA}\}_{t\ge0}$ and the nonlinear operates $b$ and $B$, for example
   see \cite{B1}, \cite{DX}.

\begin{lemma}\label{lem semigroup} For the semigroup $\{e^{tA}\}_{t\ge0}$, we have:
\begin{enumerate}
\item[(1)] for any $\theta\leq \gamma, x\in \HH_{\theta}$,
$$
\left\|e^{tA}x\right\|_\gamma\leq C t^{-\frac{\gamma-\theta}{2}}\|x\|_\theta;
$$
\item[(2)] for any $\sigma\in[0,1]$ there exists $C_{\sigma}>0$ such that for any $0<s<t$ and $x\in \HH$,
$$
\left|e^{tA}x-e^{sA}x \right|\le C_{\sigma}\frac{(t-s)^{\sigma}}{s^{\sigma}} |x|;
$$
\item[(3)] for any $\sigma\in[0,2]$ there exists $C_{\sigma}>0$ such that for any $0\leq s<t$ and $x\in \HH_{\sigma}$,
$$
\left|e^{tA}x-e^{sA}x \right|\le C_{\sigma}(t-s)^{\sigma/2} \|x\|_{\sigma}.
$$
\end{enumerate}
\end{lemma}

\begin{lemma} \label{Property B0}
For any $x, y \in \VV$,
$$ b(x,x,y)=-b(x,y,x),\quad b(x,x,x)=0.$$
\end{lemma}

\begin{lemma} \label{Property B1} Suppose $\alpha_{i}\geq 0~(i=1,2,3)$ satisfies one of the following conditions:
\begin{enumerate}
\item[(1)] $\alpha_{i}\neq\frac{1}{2}(i=1,2,3), \alpha_{1}+\alpha_{2}+\alpha_{3}\geq \frac{1}{2}$;
\item[(2)] $\alpha_{i}=\frac{1}{2}$ for some $i$, $\alpha_{1}+\alpha_{2}+\alpha_{3}>\frac{1}{2}$,
\end{enumerate}
then $b$ is continuous from $\HH_{\alpha_{1}}\times \HH_{\alpha_{2}+1}\times \HH_{\alpha_{3}}$ to $\mathbb{R}$, i.e.
$$\big|b(x,y,z)\big|\leq C\|x\|_{\alpha_{1}}\cdot\|y\|_{\alpha_{2}+1}\cdot\|z\|_{\alpha_{3}}.$$

\end{lemma}
The following inequalities can be derived by the above lemma.
\begin{corollary} \label{Property B3} For any $x\in \VV$, we have:
\begin{enumerate}
\item[(1)]
$ |B(x)|\leq C\|x\|^{2}$;
\item[(2)]
$\|B(x)\|_{-1}\leq C|x|\cdot\|x\|.$
\end{enumerate}
\end{corollary}

\begin{lemma} \label{Property B2} For any $x,y\in \VV$, we have:
\begin{enumerate}
\item[(1)]
$ |B(x)-B(y)|\leq C\|x-y\|(\|x\|+\|y\|)$;
\item[(2)] $\|B(x)-B(y)\|_{-1}\leq C|x-y|\left(\|x\|+\|y\|\right)$.
\end{enumerate}
\end{lemma}

\vspace{0.3cm}
\textbf{Acknowledgment}. This work was conducted during the first, second and fourth authors visited
the Department of Mathematics, Faculty of Science and Technology, University
of Macau, and they thank for the finance support/hospitality. Xiaobin Sun is
supported by the  NNSFC(11601196, 11771187, 11931004),
Natural Science Foundation of the Higher Education Institutions of Jiangsu Province (16KJB110006) and the Project Funded by the Priority Academic Program Development of Jiangsu
Higher Education Institutions.  Ran Wang is supported by  NNSFC (11871382). Lihu Xu is supported by the following grants: NNSFC(11571390), Macau S.A.R. FDCT 030/2016/A1 and FDCT 038/2017/A1, University of Macau MYRG (2015-00021-FST, 2016-00025-FST). Xue Yang is supported by  NNSFC (11861029, 11771329, 11871052).

\vskip0.5cm
\end{document}